\documentclass[12pt]{article}
\usepackage[utf8]{inputenc}
\usepackage{mathtools,xcolor,pdflscape}

\usepackage{amsfonts}
\usepackage{amsmath,latexsym,amssymb,amsthm,enumerate,amsmath,amscd}
\usepackage[mathscr]{eucal}
\usepackage{bbm,color,cancel}
\usepackage{tikz}

\usepackage{comment}
\usepackage{array}
\newtheorem{thm}{Theorem}
\newtheorem{cor}{Corollary}
\theoremstyle{plain}
\newtheorem{theorem}{Theorem}[section]
\newtheorem{proposition}[theorem]{Proposition}
\newtheorem{corollary}[theorem]{Corollary}
\newtheorem{lemma}[theorem]{Lemma}
\newtheorem*{ack}{Acknowledgment}
\theoremstyle{definition}
\newtheorem{definition}[theorem]{Definition}

\newtheorem{remark}[theorem]{Remark}
\newtheorem{conclusion}[theorem]{Conclusion}

\newtheorem{example}[theorem]{Example}

\newtheorem{question}[theorem]{Question}

\def\cha{\mathrm{char}\ }
\def\Proj{\mathrm{Proj}\ }
\def\Spech{\mathrm{Spech}}

\def\Pgl{\mathrm{PGL}}

\def\Hilb{\mathrm{Hilb}}

\def\Stab{\mathrm{Stab}}
\def\Gor{\mathrm{Gor}}

\def\Grass{\mathrm{Grass}}

\def\cod{\mathrm{cod}}

\def\<{\left<}
\def\>{\right>}

\def\F{{\sf k}}

\def\G{\mathcal{G}}

\def\ns{\footnotesize \it}

\def\Z{\mathfrak{Z}}

\def\Ext{\mathrm{Ext}}

\def\CI{\mathrm{CI}}

\def\m{\mathfrak{m}}

\def\Ann{\mathrm{Ann}}

\newcommand\blfootnote[1]{
  \begingroup
  \renewcommand\thefootnote{}\footnote{#1}
  \addtocounter{footnote}{-1}
  \endgroup
}

\title{Limits of graded Gorenstein algebras of Hilbert function $(1,3^k,1)$}
\date{Revised September 12, 2023}
\begin{document}
\author{
	Nancy Abdallah\\[0.05in]{\ns University of Bor\r{a}s, Bor\r{a}s, Sweden.}\\[0.2in]
	Jacques Emsalem\footnote{Our friend and colleague 
	Jacques Emsalem passed on 12 July, 2022, after an illness; he had participated actively and enthusiastically in our work.}\\[0.05in]{\ns Paris, France.}\\[0.2in]
Anthony Iarrobino\\[.05in]
{\ns Department of Mathematics, Northeastern University, Boston, MA 02115,
USA.}\\[0.2in]Joachim Yam\'{e}ogo\\[0.05in]{\ns Universit\'e C\^ote d'Azur, CNRS, LJAD, France
}}
\maketitle

\abstract Let $R={\sf k}[x,y,z]$, the polynomial ring over a field $\sf k$.  Several of the authors previously classified nets of ternary conics and their specializations over an
algebraically closed field \cite{AEI}. We here show that when $\sf k$ is algebraically closed, and considering the Hilbert function sequence $T=(1,3^k,1), k\ge 2$ (i.e. $T=(1, 3,3,\ldots, 3,1)$ where $k$ is the multiplicity of $3$), then the family $G_T$ parametrizing graded Artinian algebra quotients  $A=R/I$ of $R$ having Hilbert function $T$ is irreducible, and $G_T$ is the closure of the family $\Gor(T)$ of Artinian Gorenstein algebras of Hilbert function $T$. We then classify up to isomorphism
the elements of these families $\Gor(T)$ and of $G_T$. Finally, we give examples of codimension three Gorenstein sequences, such as $(1,3,5,3,1)$, for which
$G_T$ has several irreducible components, one being the Zariski closure of $\Gor(T)$.

\renewcommand\footnotemark{}
\section{Introduction.} 
Let $R={\sf k}[x_1,\ldots,x_r]$ be a standard-graded polynomial ring over a field $\sf k$ of arbitrary characteristic (unless we otherwise specify) and denote by $S={\sf k}_{DP}[X_1,\ldots, X_r]$ the corresponding ring of divided powers, upon which $R$ acts by contraction: that is, 
\begin{equation}\label{contracteq}
x_i^k\circ X_j^u=\delta_{i,j}X_j^{u-k} \text { if $u\ge k$ or $0$ if $u<k$},
\end{equation}
where $\delta_{i,j}$ is the Kronecker delta.
We have $R=\oplus_{i=0}^\infty R_i$ the sum of its homogeneous degree-$i$ components.\par
 Let $T$ be a \emph{Gorenstein sequence},  a sequence that may occur as the Hilbert function $H(A)$ of a graded Artinian Gorenstein quotient $A=R/I$. The family of graded Gorenstein algebras $\Gor(T)$ of Hilbert function $T$ is an open subfamily of $G_T$, the family of all graded algebra quotients of $R$ having Hilbert function $T$.\par
For codimension $r=2$ the Gorenstein algebras $A$ are complete intersections. Also, in codimension 2 for $\sf k$ algebraically closed, we have that $\Gor(T)$ is irreducible of known dimension and is open dense in $G_T$  \cite{I1}. For codimension $r=3$, the Gorenstein sequences $T$ are known as a consequence of the D. Buchsbaum-D. Eisenbud Pfaffian resolution theorem
\cite{BuEi, St}; S.J. Diesel showed also for $r=3$ that over an algebraically closed field the variety $\Gor(T)$ is irreducible; and she determined a poset of minimal resolution strata, and their associated dimensions \cite{Di}. A consequence for $T=(1,3,3,1)$ is that
 $\Gor(T)=\overline{CI_T}$ is the closure of the family of complete instersections of Hilbert function $T$. For $T=(1,3^k,1), k\ge 3$ every $A\in \Gor(T)$ has five generators. We consider the\vskip 0.2cm\noindent
 {\bf Question.} Given $T$ a Gorenstein sequence of codimension three, is $\Gor(T)$ open dense in $G_T$?\par
 We will show that for $T=(1,3^k,1)$ the answer is  ``Yes'' (Theorem~\ref{1thm}) but in general is ``No'' (Theorem \ref{3thm}). \par
 For codimension $r\ge 4$, the family $\Gor(T)$ itself may have several irreducible components \cite{IS}. 
 \par
 A \emph{net of conics} in the projective plane $\mathbb P^2$ is a $3$-dimensional vector space $V$ of conics, elements of $R_2$ for $r=3$. These nets are parametrized by the Grassmanian $\Grass(R_2,3)$. The projective linear group $\Pgl(3)$ acts on $R_1=\langle x,y,z\rangle$, hence on $R={\sf k}[x,y,z]$.
 In \cite{AEI} the first three authors determined the isomorphism classes under $\Pgl(3)$ of nets of conics over an algebraically closed field $\sf k$ not of characteristic $2$ or $3$,  and as well their specializations. These are set out in our  Figure~\ref{1fig} below, based on \cite[Tables~1,5]{AEI}).\footnote{These classes had been determined also by
 C.T.C. Wall - who considered as fields both the complexes and the reals, and also gave an account of specializations \cite{Wall};  A. Conca gives further information about the nets \cite{Co}. N. Onda - apparently unaware of previous results - gave recently an independent determination of the isomorphism classes over an algebraically closed field of characteristic not $2$ or $3$, without discussion of specialization \cite{Onda}.  See \cite[Appendices  A,B]{AEI} for some historical account, and discussion of related problems, in particular \cite[Table 10, Appendix B]{AEI} for a concordance of the results of \cite{Co,Wall, AEI} classifying nets of conics.}  A detailed explanation of Figure \ref{1fig} is in Section \ref{netssec}, page \pageref{organize1fig} below - this is needed to fully understand the references to it in the Introduction.
\par We will denote by $\Gor_{\#6a}(T)$, for example, the family of Artinian Gorenstein algebras $A=R/I$ of Hilbert function $T$, such that the net $V=I_2$ is in the isomorphism class \#6a of Figure \ref{1fig}, that is $I_2\cong \langle xy,xz,yz\rangle$ under $\Pgl(3)$-action. Equivalently, $\Gor_{\#6a}(T)$ parametrizes those  $A=R/I$ for which $I\cong\Ann F, F=X^j+\alpha_1Y^j+\alpha_2Z^j$, for some $\alpha_1,\alpha_2\in \sf k$.\footnote{When $\sf k$ is closed under $j$-th roots, we may assume that $F\cong X^j+Y^j+Z^j$ after a coordinate change replacing $Y$ by $\alpha_1^{1/j}Y$, $Z$ by $\alpha_2^{1/j}Z$.}
 \vskip 0.2cm\noindent
 We recall from \cite[Section 6.3]{AEI},
 the \emph{scheme length} $\ell(V)$ of a net of conics $V$ is the length of the scheme $\Proj (R/(V))$; when finite this is just the integer $a$ for which the Hilbert function 
 \begin{equation}\label{leqn}
 H(R/(V))=(1,3,3,\ldots,\overline{a})
 \end{equation} has an infinite tail $\overline{a}=(a,a,\ldots,a,\ldots)$ of $a$'s. The scheme $\Z_V$ determined by $V$ is empty if $\ell(V)=0$; it is a finite set of $\ell(V)$ points (counted with their multiplicities if a point is non-reduced) when $1\le \ell(V)\le 3$; and is a line in the case $V\cong \langle x^2,xy,xz\rangle$ (\#2a in Figure \ref{1fig}) - we say then $\ell(V)=\infty$. \par

Our first main result is 
 \begin{thm}\label{1thm}  Let $T=(1,3^k,1), k\ge 2$. When $ k=2$, the family $G_T$ of graded quotients of $R={\sf k}[x,y,z]$ having Hilbert function $T$, is the Zariski closure of the family $CI_T$ of graded complete intersection quotients of $R$ having Hilbert function $T$; and $G_T$ has dimension $9$.\par When $k\ge 3$ then $G_T$ is in the closure of $\Gor_{\#6a}(T)$, and $G_T$ has dimension eight.
 \par In each case $G_T\subset \overline{\Gor(T)}$ and is irreducible.
 \end{thm}\noindent
 We show this below in Theorem \ref{schl3thm} (scheme length of $V=I_2$ is three), Theorems \ref{k=2GTthm} (case $k=2$) and \ref{kge3thm} (case $k\ge 3$).\vskip 0.2cm
 
In our second main result we show that the pattern of specializations of algebras $A=R/I$ having Hilbert function $T$ corresponds to that for the corresponding nets of conics $V=I_2$,  provided we stay within constant scheme-length families of conics. In the following Theorem $V,V^\prime$ are nets as in Figure \ref{1fig} with $V^{\prime}$ specializing to $V$ according to the arrows of Figure \ref{1fig}.
 \begin{thm}\label{2thm} Assume that the net $V\in \Grass(R_2,3)$ is in the Zariski closure of nets isomorphic to $V^\prime$, and that the scheme length $\ell(V)=\ell(V^\prime)$ and neither $V$ nor $V^\prime$ is \#2a. Let $A=R/I$ be an algebra of Hilbert function $T=(1,3^k,1)$ such that $I_2\cong V$. Then $A$ is in the closure of a family $A(x)=R/I(x)$ of algebras having $I_2(x)\cong V^\prime$, that is $A=\lim_{x\to x_0} A(x), x\in X\backslash x_0$. We may take $X$ to be a curve.
 \end{thm}
 We show Theorem \ref{2thm} in Theorem \ref{schl3thm} (case scheme length of $V$ is three - including the case $k\ge 3$), and in Theorem \ref{scheme2deformthm} (case $k=2$).
 Remark \ref{nondeformrem} shows that an analogous statement to Theorem~\ref{2thm} does not hold when the scheme lengths of $V, V^\prime$ are not the same.\par
We show in Theorem \ref{small2compprop},
\begin{thm}[Smallest $G_T$ not in $\overline{\Gor(T)}$]\label{3thm} Let $5\le a\le 6$ and $T=(1,3,a,3,1)$, 
 or let $T=(1,3,6,6,3,1)$. Then $G_T$ has two irreducible components, $\overline{\Gor(T)}$, and a second, larger component $U_T$ (Definition~\ref{UTdef}), whose dimensions are given in Figure \ref{compfig}. 
 \end{thm}
 In Section \ref{isomsec} we study $G_T$: we classify graded algebras of Hilbert function $T=(1,3^k,1)$ up to isomorphism, over an algebraically closed field $\sf k$. We discuss briefly their specializations. Given a net $V\subset R_2$ we denote by $\Gor_V(T)$ the
 family of Gorenstein algebra quotients $A=R/I$ with $I_2=V$, of Hilbert function $T$; and by $G_V(T)$ the family of Artinian algebra quotients $A=R/I$ with $I_2=V$, of Hilbert function $T$. We denote by $\Gor_{\{V\}}(T)$ and
 $G_{\{V\}}(T)$ the corresponding families where $I_2\cong V$ under $\Pgl(3)$ action.\par
 In Section \ref{isomsec}, we show first the following result (Theorem 4.1). \vskip 0.2cm\noindent
 \begin{thm}[Dimension theorem]\label{4thm}   Let $T=(1,3^k,1), k\ge 2$. Then $G_V(T)$ and $G_{\{V\}}(T)$ are irreducible.
 We let $d= d(V)$ be the dimension of the family of nets of conics that are $\Pgl(3)$ equivalent to $V$.\begin{enumerate}[(i).]
\item Let $T=(1,3,3,1).$  Then the family of Artinian algebras $G_{\{V\}}(T)$ has dimension $d$ if $\ell=0$, dimension $d+\ell-1$ if $1\le \ell\le 3$; and has dimension $5$ if $\ell=\infty$ (the case $V=\#$2a).
\item Let $T=(1,3^k,1)$ with $k\ge 3$. Then  $G_{\{V\}}(T)$ has dimension $d+2$ if $\ell=3$ and dimension $7$ if 
$\ell=\infty$ (case V=\#2a), and is empty otherwise.
\end{enumerate}
 \end{thm}
 We then show our main classification result for Artinian quotients of $R$. Recall that the nets of scheme length zero determine CI algebras: these are the nets of \#8b, and the three nets \#8c,\#7c, and \#6d.
 
 \begin{thm}\label{5thm} Let $T=(1,3^k,1), k\ge 2$. We determine the isomorphism classes of algebras in $G_V(T)$, the family of Artinian quotients $A=R/I$ in $G_T$ for which $I_2=V$. \par
 (i). The number of classes for $V$ of scheme lengths $\ell=2$ and $3$ satisfy
 \end{thm}
 \qquad\qquad\qquad \qquad\qquad Cases $\ell=2$.  \qquad\quad$\, \vert$\quad \qquad Cases $\ell=3$.\vskip 0.2cm\par
 $\begin{array}{c||cccc|cccc}
 \text{ Class of $I_2$}&\#7a&\#6b&\#6c&\#5b&\#6a&\#5&\#4&\#2b\\\hline
 \text{ \# Isom classes} &6&\infty&3&3 &3&7&3&3.\\
  \end{array}
 $\vskip 0.2cm\par
 There are among the nets $V$ of scheme length 3, exactly four isomorphism classes of Gorenstein algebras in $G_V(T)$, namely
 \#6a.i,\#5a.i, \par\noindent \#5a.ii, and \#4.iii., (see the Betti C table in Figure \ref{BettiCfig}). There are no further non-CI Gorenstein classes of algebras in $G_T$.\footnote{The notation such as \#6a.i for the isomorphism classes is found in Theorem \ref{GTclasseslength3} and the equations listed there.}
 \vskip 0.3cm
 (ii). Classes for the $V=xR_1$ (\#2a) of scheme length $\infty$:\par
 When $T=(1,3,3,1)$: there are 3 classes. For $T=(1,3^k,1), k\ge 3$ \vskip 0.2cm
 and  $I=(V,f,f_5,f_6)$ with $ f\in I_3; f_5,f_6\in I_j$ we have
  for\par $ f=yz(y-z):   \infty^2, \infty$; for $  f=yz^2$:  5 classes; for $f=z^3$: 4 classes.
\vskip 0.2cm\par\noindent
{\bf Note.} Here $\infty^2$ means a two-parameter family of classes, see Corollary \ref{6cor}. Theorem \ref{5thm} summarizes the results of Theorems \ref{sch2algthm}, \ref{GTclasseslength3} and \ref{2aAthm} for $V$ of scheme length two, three and $\infty$, respectively. We also specify the Betti tables for the classes for $V$ of scheme lengths two and three.\vskip 0.2cm
  \par Recall that the family of nets having the form \#8b in Figure \ref{1fig} is a continuous family of complete intersections, whose isomorphism classes are determined by their ${\sf j}$-invariant ${\sf j}(\lambda)$ (see \cite[Footnote 6]{AEI}).\vskip 0.2cm\par\noindent
 {\bf Question.} Fix any net $V\subset R_3$. Do we have continuous families of Artinian algebras $A=R/I$ in $G_T$ with $I_2=V$? \par We answer this:\par
 \setcounter{cor}{5}
 \begin{cor}\label{6cor}  Aside of the complete intersection class of nets \#8b, there are three continuous families of isomorphism classes in these $G_T$. First, when $T=(1,3,3,1)$ and $V\cong \langle x^2,xy,yz+z^2\rangle$ (\#6b),  Then, when $T=(1,3^k,1), k\ge 3$, $V\cong xR_1$ (\#2a),  and $I$ contains $f=yz(y-z)$ we have the two continuous families $I(a)$ and $I(a,b)$ where $I(a,b)=I(a^{-1},a^{-1}b)$ if $a\not=0$ (Equations \eqref{Wab2eq}, \eqref{Wabsigmaeq},\eqref{waeqn}). For all other nets $V$ there are only a finite number of non-isomorphic Artinian algebras having $I_2=V$.
 \end{cor}
 \begin{proof} See Conclusion \ref{6bconclude} for \# 6b and Theorem \ref{2aAthm} for \#2a.  
 \end{proof}
 In Section \ref{prelimsec} we recall some definitions and results about nets of conics. Section 3 is devoted to the deformations of graded algebras to graded Gorenstein algebras; Section \ref{isomsec} is devoted to the classification up to isomorphism of graded algebra quotients of $R$ having Hilbert function $T=(1,3^k,1), k\ge 2$.
 In Section \ref{bettisec} we consider some questions about limits of Betti tables.
\subsubsection{Glossary of terms}
 We list some key terms used. 
 All algebras considered are graded.
  \begin{itemize}
\item A Hilbert function $T$, usually of the form $T=(1,3^k,1)=(1,3,3,\cdots, 3,1)$ with $k$ occurences of $3$.
\item The family $\Gor(T)$ of graded Artinian Gorenstein (AG) algebras $A$, quotients of $R={\sf k}[x,y,z]$, having Hilbert function $H(A)=T$.  This is known to be irreducible.
\item The family $G_T$ of all graded quotients of $R$ having Hilbert function $T$.
\item $\Gor_{\# t}$: the family of AG algebras associated with the entry $t$ of Figure \ref{1fig}.
\item $\CI_T$ when $k=2$: the family of complete intersection quotients of $R$ having Hilbert function $T$.
\item Let $V$ be a net of conics in $R_2$. Then we have (Definition \ref{GVdef}) \par
- $\Gor_V(T)$ is the family of graded AG quotients $A=R/I$ with $I_2=V$ and Hilbert function $H(A)=T$.\par
- $G_V(T)$ is the family of graded Artinian algebra quotients of $R$ having $I_2=V$ and Hilbert function $T$.\par
- $\Gor_{\{V\}}(T)$ and $G_{\{V\}}(T)$ are the corresponding families where $I_2=V$ under $\Pgl(3)$ action.
\end{itemize}
 For example $\Gor_{6a}(T)=\Gor_{\{V\}}(T)$ for $ V=\langle xy, yz,xz \rangle$ (\#6a of Figure \ref{1fig}).
  \tableofcontents
\section{Preliminaries.}\label{prelimsec}
\subsection{Nets of conics.}\label{netssec}
We adapt these notions from \cite{AEI}. Figure \ref{1fig} below combines information
from  \cite[Tables 1,5]{AEI}. Recall that $\Pgl(3)$ acts on the $\sf k$-vector space $\langle x,y,z\rangle$, hence on $R={\sf k}[x,y,z]$. The following definition is so the reader can connect better the Figure \ref{1fig} here with \cite[Table 1]{AEI}; the concepts such as ``associated cubic'' are not essential here - we include them to make the connection with previous work, with the idea they might be potentially relevant to a reader.
\begin{definition} Let $V\subset R_2$ be a subvector space of $\sf k$-dimension three.
The \emph{scheme length} $\ell(V)$ of $V$ is the length of the scheme $\mathfrak Z_V=\Proj(R/(V))$, or zero if $R/(V)$ is a complete intersection (CI), or $\infty$ if $V\cong xR_1$, in which case the associated scheme $\Z_V$ is a line.\par
The \emph{associated cubic} $\Gamma(V)$ of the net of conics $V=\langle v_1,v_2,v_3\rangle$ is the cubic
 $f(a_1,a_2,a_3)$ such that $f=0$ is the locus of elements $a_1v_1+a_2v_2+a_3v_3$ that are decomposable - that is - may be factored into a product of linear forms.
 Here $a_1,a_2,a_3\in R_1$ - are linear.\par
  The \emph{double line} subset $D_2(V)\subset \Gamma(V)$ denotes the set of perfect squares, that is the elements of V having the form $L^2$ for some $L\in R_1$.
 \end{definition}\vskip 0.2cm\noindent
    {\it Organization of  Figure \ref{1fig}:}\par \label{organize1fig}
   In referencing the isomorphism classes of nets in Figure \ref{1fig}, we will use \#7a,\#7b,\#7c to denote respectively the nets having orbit dimension 7, leftmost, middle, and right - so \#7a is the net $\langle
x^2+yz,xy,xz\rangle$, or, used loosely, any net that is $\Pgl(3)$-equivalent to it.~\footnote{The nets in \#8b correspond to $J\phi$ where $\phi$ is a smooth cubic not isomorphic to the harmonic cubic $x^3+y^3+z^3$. Here $\omega$ is a primitive cube root of one. The omitted values $-1,-\omega,-\omega^2$ correspond to singular cubics $\phi$  giving rise to the net in \#6a. The omitted values $0,2,2\omega,2\omega^2$ give a cubic isomorphic to the harmonic cubic. The ${\sf j}(\lambda)$ in
Figure \ref{1fig} \#8b is the ${\sf j}$-invariant giving the isomorphism class of the cubic.  For further details see \cite[Table 1]{AEI}.}\vskip 0.2cm\par\noindent

 In Figure \ref{1fig} we give for each $\Pgl(3)$ orbit of nets of conics
\begin{enumerate}[i.]
\item a basis for a net belonging to it,
\item a diagram indicating the isomorphism type of its associated cubic $\Gamma(V)$, where each point of $\Gamma(V)\cap D_2$ is marked by a circle around it,
\item the scheme length $\ell(V)$ for $V$ in the orbit, as an italicized integer $\it 0,1,2$ or $\it 3$ to the right of the diagram, or as $\infty$ for the net $V$ of \#2a. (This information is not in \cite[Table 1]{AEI}, and we use scheme length extensively here).
\item the dimension of the orbit is indicated by the integer in the left column. 
\end{enumerate}
The arrows oriented downward $V\to V'$ indicate that the orbit of $V'$ is included in the Zariski closure of the orbit of $V$.  The arrows downward from the one-parameter family $V_{\sf j}\to V'$ indicate that the orbit of $V'$ is in the closure of each orbit $V_{\sf j}$, with $\sf j$, the invariant of the cubic, held constant. See \cite[Footnote 6]{AEI} for the $\sf j$ invariant in terms of $\lambda$. \vskip 0.2cm

 \begin{figure}[p]
 \vskip -1.2cm
\tikzset{every picture/.style={line width=0.75pt}}         

\begin{tikzpicture}[x=0.75pt,y=0.75pt,yscale=-0.9,xscale=0.9]

\draw  [color={rgb, 255:red, 0; green, 0; blue, 0 }  ][line width=0.75] [line join = round][line cap = round] (118.5,46.63) .. controls (108.2,54.88) and (92.71,69.4) .. (81.5,75.63) .. controls (74.02,79.79) and (62.94,77.28) .. (61.5,68.63) .. controls (60.74,64.05) and (64.91,57.47) .. (69.5,56.63) .. controls (77.34,55.21) and (98.25,61.88) .. (106.5,64.63) .. controls (109.05,65.48) and (120.5,68.03) .. (120.5,69.63) ;
\draw    (87.78,108.82) -- (87.07,177.48) ;
\draw [shift={(87.05,179.48)}, rotate = 270.59] [color={rgb, 255:red, 0; green, 0; blue, 0 }  ][line width=0.75]    (10.93,-3.29) .. controls (6.95,-1.4) and (3.31,-0.3) .. (0,0) .. controls (3.31,0.3) and (6.95,1.4) .. (10.93,3.29)   ;
\draw   (66,219.32) .. controls (66,205.51) and (77.19,194.32) .. (91,194.32) .. controls (104.81,194.32) and (116,205.51) .. (116,219.32) .. controls (116,233.12) and (104.81,244.32) .. (91,244.32) .. controls (77.19,244.32) and (66,233.12) .. (66,219.32) -- cycle ;
\draw    (31,222.57) -- (150.5,223.07) ;
\draw    (86.78,276.82) -- (87.04,339.48) ;
\draw [shift={(87.05,341.48)}, rotate = 269.76] [color={rgb, 255:red, 0; green, 0; blue, 0 }  ][line width=0.75]    (10.93,-3.29) .. controls (6.95,-1.4) and (3.31,-0.3) .. (0,0) .. controls (3.31,0.3) and (6.95,1.4) .. (10.93,3.29)   ;
\draw    (78.72,351.7) -- (122.72,421.7) ;
\draw    (49.37,412.7) -- (130.72,412.7) ;
\draw    (53.37,421.7) -- (91.37,351.7) ;
\draw   (305.23,64.2) .. controls (305.23,55.43) and (312.34,48.32) .. (321.12,48.32) .. controls (329.89,48.32) and (337,55.43) .. (337,64.2) .. controls (337,72.97) and (329.89,80.08) .. (321.12,80.08) .. controls (312.34,80.08) and (305.23,72.97) .. (305.23,64.2) -- cycle ;
\draw    (343.9,107.08) .. controls (383.9,77.08) and (312.9,48.08) .. (352.9,18.08) ;
\draw  [color={rgb, 255:red, 0; green, 0; blue, 0 }  ][line width=0.75] [line join = round][line cap = round] (587.5,46.63) .. controls (577.2,54.88) and (561.71,69.4) .. (550.5,75.63) .. controls (543.02,79.79) and (531.94,77.28) .. (530.5,68.63) .. controls (529.74,64.05) and (533.91,57.47) .. (538.5,56.63) .. controls (546.34,55.21) and (567.25,61.88) .. (575.5,64.63) .. controls (578.05,65.48) and (589.5,68.03) .. (589.5,69.63) ;
\draw   (563.23,62.7) .. controls (563.23,59.17) and (566.09,56.32) .. (569.62,56.32) .. controls (573.14,56.32) and (576,59.17) .. (576,62.7) .. controls (576,66.23) and (573.14,69.08) .. (569.62,69.08) .. controls (566.09,69.08) and (563.23,66.23) .. (563.23,62.7) -- cycle ;
\draw  [color={rgb, 255:red, 0; green, 0; blue, 0 }  ][line width=0.75] [line join = round][line cap = round] (306.53,223.98) .. controls (318.24,223.98) and (329.89,216.2) .. (336.15,210.96) .. controls (339.61,208.07) and (348.19,201.35) .. (348.19,199.18) ;
\draw   (298.83,222.39) .. controls (299.21,217.38) and (303.58,213.63) .. (308.59,214.01) .. controls (313.59,214.4) and (317.34,218.77) .. (316.96,223.78) .. controls (316.58,228.78) and (312.21,232.53) .. (307.2,232.15) .. controls (302.19,231.76) and (298.44,227.39) .. (298.83,222.39) -- cycle ;
\draw   (535,219.32) .. controls (535,205.51) and (546.19,194.32) .. (560,194.32) .. controls (573.81,194.32) and (585,205.51) .. (585,219.32) .. controls (585,233.12) and (573.81,244.32) .. (560,244.32) .. controls (546.19,244.32) and (535,233.12) .. (535,219.32) -- cycle ;
\draw    (499.05,221.82) -- (619.05,220.82) ;
\draw   (253.28,410.32) .. controls (253.28,406.79) and (256.14,403.93) .. (259.67,403.93) .. controls (263.19,403.93) and (266.05,406.79) .. (266.05,410.32) .. controls (266.05,413.84) and (263.19,416.7) .. (259.67,416.7) .. controls (256.14,416.7) and (253.28,413.84) .. (253.28,410.32) -- cycle ;
\draw   (529.62,221.32) .. controls (529.62,217.79) and (532.47,214.93) .. (536,214.93) .. controls (539.53,214.93) and (542.38,217.79) .. (542.38,221.32) .. controls (542.38,224.84) and (539.53,227.7) .. (536,227.7) .. controls (532.47,227.7) and (529.62,224.84) .. (529.62,221.32) -- cycle ;
\draw   (577.62,221.32) .. controls (577.62,217.79) and (580.47,214.93) .. (584,214.93) .. controls (587.53,214.93) and (590.38,217.79) .. (590.38,221.32) .. controls (590.38,224.84) and (587.53,227.7) .. (584,227.7) .. controls (580.47,227.7) and (577.62,224.84) .. (577.62,221.32) -- cycle ;
\draw   (235,384.32) .. controls (235,370.51) and (246.19,359.32) .. (260,359.32) .. controls (273.81,359.32) and (285,370.51) .. (285,384.32) .. controls (285,398.12) and (273.81,409.32) .. (260,409.32) .. controls (246.19,409.32) and (235,398.12) .. (235,384.32) -- cycle ;
\draw    (199.05,410.82) -- (319.05,409.82) ;
\draw    (326.78,154.82) -- (327.03,186.48) ;
\draw [shift={(327.05,188.48)}, rotate = 269.55] [color={rgb, 255:red, 0; green, 0; blue, 0 }  ][line width=0.75]    (10.93,-3.29) .. controls (6.95,-1.4) and (3.31,-0.3) .. (0,0) .. controls (3.31,0.3) and (6.95,1.4) .. (10.93,3.29)   ;
\draw    (554.72,349.7) -- (598.72,419.7) ;
\draw    (525.37,410.7) -- (606.72,410.7) ;
\draw    (529.37,419.7) -- (567.37,349.7) ;
\draw    (558.78,105.82) -- (559.04,175.48) ;
\draw [shift={(559.05,177.48)}, rotate = 269.79] [color={rgb, 255:red, 0; green, 0; blue, 0 }  ][line width=0.75]    (10.93,-3.29) .. controls (6.95,-1.4) and (3.31,-0.3) .. (0,0) .. controls (3.31,0.3) and (6.95,1.4) .. (10.93,3.29)   ;
\draw    (417.43,352.07) -- (417.43,422.07) ;
\draw    (375.62,376.98) -- (460.78,377.05) ;
\draw   (381.14,381.3) .. controls (381.53,376.3) and (385.9,372.55) .. (390.91,372.93) .. controls (395.91,373.32) and (399.66,377.69) .. (399.28,382.69) .. controls (398.89,387.7) and (394.52,391.45) .. (389.52,391.07) .. controls (384.51,390.68) and (380.76,386.31) .. (381.14,381.3) -- cycle ;
\draw   (434.98,381.3) .. controls (435.36,376.3) and (439.73,372.55) .. (444.74,372.93) .. controls (449.75,373.32) and (453.5,377.69) .. (453.11,382.69) .. controls (452.73,387.7) and (448.36,391.45) .. (443.35,391.07) .. controls (438.34,390.68) and (434.59,386.31) .. (434.98,381.3) -- cycle ;
\draw   (555.62,360.32) .. controls (555.62,356.79) and (558.47,353.93) .. (562,353.93) .. controls (565.53,353.93) and (568.38,356.79) .. (568.38,360.32) .. controls (568.38,363.84) and (565.53,366.7) .. (562,366.7) .. controls (558.47,366.7) and (555.62,363.84) .. (555.62,360.32) -- cycle ;
\draw   (587.62,410.32) .. controls (587.62,406.79) and (590.47,403.93) .. (594,403.93) .. controls (597.53,403.93) and (600.38,406.79) .. (600.38,410.32) .. controls (600.38,413.84) and (597.53,416.7) .. (594,416.7) .. controls (590.47,416.7) and (587.62,413.84) .. (587.62,410.32) -- cycle ;
\draw   (528.62,410.32) .. controls (528.62,406.79) and (531.47,403.93) .. (535,403.93) .. controls (538.53,403.93) and (541.38,406.79) .. (541.38,410.32) .. controls (541.38,413.84) and (538.53,416.7) .. (535,416.7) .. controls (531.47,416.7) and (528.62,413.84) .. (528.62,410.32) -- cycle ;
\draw    (113.78,109.82) -- (282.05,205.83) ;
\draw [shift={(283.78,206.82)}, rotate = 209.71] [color={rgb, 255:red, 0; green, 0; blue, 0 }  ][line width=0.75]    (10.93,-3.29) .. controls (6.95,-1.4) and (3.31,-0.3) .. (0,0) .. controls (3.31,0.3) and (6.95,1.4) .. (10.93,3.29)   ;
\draw    (540.78,105.82) -- (365.52,205.83) ;
\draw [shift={(363.78,206.82)}, rotate = 330.29] [color={rgb, 255:red, 0; green, 0; blue, 0 }  ][line width=0.75]    (10.93,-3.29) .. controls (6.95,-1.4) and (3.31,-0.3) .. (0,0) .. controls (3.31,0.3) and (6.95,1.4) .. (10.93,3.29)   ;
\draw    (109.78,278.82) -- (220.11,351.71) ;
\draw [shift={(221.78,352.82)}, rotate = 213.45] [color={rgb, 255:red, 0; green, 0; blue, 0 }  ][line width=0.75]    (10.93,-3.29) .. controls (6.95,-1.4) and (3.31,-0.3) .. (0,0) .. controls (3.31,0.3) and (6.95,1.4) .. (10.93,3.29)   ;
\draw    (316.78,277.82) -- (268.92,347.17) ;
\draw [shift={(267.78,348.82)}, rotate = 304.61] [color={rgb, 255:red, 0; green, 0; blue, 0 }  ][line width=0.75]    (10.93,-3.29) .. controls (6.95,-1.4) and (3.31,-0.3) .. (0,0) .. controls (3.31,0.3) and (6.95,1.4) .. (10.93,3.29)   ;
\draw    (347.78,279.05) -- (402.56,350.23) ;
\draw [shift={(403.78,351.82)}, rotate = 232.42] [color={rgb, 255:red, 0; green, 0; blue, 0 }  ][line width=0.75]    (10.93,-3.29) .. controls (6.95,-1.4) and (3.31,-0.3) .. (0,0) .. controls (3.31,0.3) and (6.95,1.4) .. (10.93,3.29)   ;
\draw    (129,279.32) -- (385.86,354.26) ;
\draw [shift={(387.78,354.82)}, rotate = 196.26] [color={rgb, 255:red, 0; green, 0; blue, 0 }  ][line width=0.75]    (10.93,-3.29) .. controls (6.95,-1.4) and (3.31,-0.3) .. (0,0) .. controls (3.31,0.3) and (6.95,1.4) .. (10.93,3.29)   ;
\draw    (561.78,276.82) -- (562.04,339.48) ;
\draw [shift={(562.05,341.48)}, rotate = 269.76] [color={rgb, 255:red, 0; green, 0; blue, 0 }  ][line width=0.75]    (10.93,-3.29) .. controls (6.95,-1.4) and (3.31,-0.3) .. (0,0) .. controls (3.31,0.3) and (6.95,1.4) .. (10.93,3.29)   ;
\draw    (546.78,276.82) -- (435.45,350.71) ;
\draw [shift={(433.78,351.82)}, rotate = 326.43] [color={rgb, 255:red, 0; green, 0; blue, 0 }  ][line width=0.75]    (10.93,-3.29) .. controls (6.95,-1.4) and (3.31,-0.3) .. (0,0) .. controls (3.31,0.3) and (6.95,1.4) .. (10.93,3.29)   ;
\draw    (523.78,278.82) -- (296.7,347.24) ;
\draw [shift={(294.78,347.82)}, rotate = 343.23] [color={rgb, 255:red, 0; green, 0; blue, 0 }  ][line width=0.75]    (10.93,-3.29) .. controls (6.95,-1.4) and (3.31,-0.3) .. (0,0) .. controls (3.31,0.3) and (6.95,1.4) .. (10.93,3.29)   ;
\draw    (158.43,535.07) -- (158.43,605.07) ;
\draw    (137.62,559.98) -- (205.78,560.05) ;
\draw    (136.52,570.07) -- (206.35,570.07) ;
\draw   (181.98,564.3) .. controls (182.36,559.3) and (186.73,555.55) .. (191.74,555.93) .. controls (196.75,556.32) and (200.5,560.69) .. (200.11,565.69) .. controls (199.73,570.7) and (195.36,574.45) .. (190.35,574.07) .. controls (185.34,573.68) and (181.59,569.31) .. (181.98,564.3) -- cycle ;
\draw    (92.78,455.82) -- (143.67,532.15) ;
\draw [shift={(144.78,533.82)}, rotate = 236.31] [color={rgb, 255:red, 0; green, 0; blue, 0 }  ][line width=0.75]    (10.93,-3.29) .. controls (6.95,-1.4) and (3.31,-0.3) .. (0,0) .. controls (3.31,0.3) and (6.95,1.4) .. (10.93,3.29)   ;
\draw    (243.78,457.82) -- (178.13,530.33) ;
\draw [shift={(176.78,531.82)}, rotate = 312.16] [color={rgb, 255:red, 0; green, 0; blue, 0 }  ][line width=0.75]    (10.93,-3.29) .. controls (6.95,-1.4) and (3.31,-0.3) .. (0,0) .. controls (3.31,0.3) and (6.95,1.4) .. (10.93,3.29)   ;
\draw    (409.78,459.82) -- (208.67,531.15) ;
\draw [shift={(206.78,531.82)}, rotate = 340.47] [color={rgb, 255:red, 0; green, 0; blue, 0 }  ][line width=0.75]    (10.93,-3.29) .. controls (6.95,-1.4) and (3.31,-0.3) .. (0,0) .. controls (3.31,0.3) and (6.95,1.4) .. (10.93,3.29)   ;
\draw    (460.43,532.07) -- (460.43,602.07) ;
\draw    (438.52,567.07) -- (508.35,567.07) ;
\draw   (484.98,566.3) .. controls (485.36,561.3) and (489.73,557.55) .. (494.74,557.93) .. controls (499.75,558.32) and (503.5,562.69) .. (503.11,567.69) .. controls (502.73,572.7) and (498.36,576.45) .. (493.35,576.07) .. controls (488.34,575.68) and (484.59,571.31) .. (484.98,566.3) -- cycle ;
\draw    (281.78,457.82) -- (445.95,530.01) ;
\draw [shift={(447.78,530.82)}, rotate = 203.74] [color={rgb, 255:red, 0; green, 0; blue, 0 }  ][line width=0.75]    (10.93,-3.29) .. controls (6.95,-1.4) and (3.31,-0.3) .. (0,0) .. controls (3.31,0.3) and (6.95,1.4) .. (10.93,3.29)   ;
\draw    (567,461.32) -- (473.42,526.67) ;
\draw [shift={(471.78,527.82)}, rotate = 325.07] [color={rgb, 255:red, 0; green, 0; blue, 0 }  ][line width=0.75]    (10.93,-3.29) .. controls (6.95,-1.4) and (3.31,-0.3) .. (0,0) .. controls (3.31,0.3) and (6.95,1.4) .. (10.93,3.29)   ;
\draw    (427,458.32) -- (455.01,524.97) ;
\draw [shift={(455.78,526.82)}, rotate = 247.21] [color={rgb, 255:red, 0; green, 0; blue, 0 }  ][line width=0.75]    (10.93,-3.29) .. controls (6.95,-1.4) and (3.31,-0.3) .. (0,0) .. controls (3.31,0.3) and (6.95,1.4) .. (10.93,3.29)   ;
\draw    (284.62,695.98) -- (352.78,696.05) ;
\draw    (283.52,706.07) -- (353.35,706.07) ;
\draw   (306.98,700.3) .. controls (307.36,695.3) and (311.73,691.55) .. (316.74,691.93) .. controls (321.75,692.32) and (325.5,696.69) .. (325.11,701.69) .. controls (324.73,706.7) and (320.36,710.45) .. (315.35,710.07) .. controls (310.34,709.68) and (306.59,705.31) .. (306.98,700.3) -- cycle ;
\draw    (283.62,700.98) -- (351.78,701.05) ;
\draw    (175.78,639.82) -- (275.29,693.86) ;
\draw [shift={(277.05,694.82)}, rotate = 208.51] [color={rgb, 255:red, 0; green, 0; blue, 0 }  ][line width=0.75]    (10.93,-3.29) .. controls (6.95,-1.4) and (3.31,-0.3) .. (0,0) .. controls (3.31,0.3) and (6.95,1.4) .. (10.93,3.29)   ;
\draw    (458,637.32) -- (354.84,688.92) ;
\draw [shift={(353.05,689.82)}, rotate = 333.42] [color={rgb, 255:red, 0; green, 0; blue, 0 }  ][line width=0.75]    (10.93,-3.29) .. controls (6.95,-1.4) and (3.31,-0.3) .. (0,0) .. controls (3.31,0.3) and (6.95,1.4) .. (10.93,3.29)   ;
\draw  [draw opacity=0] (495.58,794.14) .. controls (495.79,795.29) and (495.93,796.48) .. (495.98,797.69) .. controls (496.51,810.29) and (487.82,820.89) .. (476.58,821.36) .. controls (465.33,821.84) and (455.78,812.01) .. (455.25,799.41) .. controls (455.15,796.95) and (455.39,794.57) .. (455.94,792.32) -- (475.61,798.55) -- cycle ; \draw   (495.58,794.14) .. controls (495.79,795.29) and (495.93,796.48) .. (495.98,797.69) .. controls (496.51,810.29) and (487.82,820.89) .. (476.58,821.36) .. controls (465.33,821.84) and (455.78,812.01) .. (455.25,799.41) .. controls (455.15,796.95) and (455.39,794.57) .. (455.94,792.32) ;  
\draw    (353,738.65) -- (461.19,781.41) ;
\draw [shift={(463.05,782.15)}, rotate = 201.57] [color={rgb, 255:red, 0; green, 0; blue, 0 }  ][line width=0.75]    (10.93,-3.29) .. controls (6.95,-1.4) and (3.31,-0.3) .. (0,0) .. controls (3.31,0.3) and (6.95,1.4) .. (10.93,3.29)   ;
\draw    (283,743.65) -- (190.93,776.48) ;
\draw [shift={(189.05,777.15)}, rotate = 340.38] [color={rgb, 255:red, 0; green, 0; blue, 0 }  ][line width=0.75]    (10.93,-3.29) .. controls (6.95,-1.4) and (3.31,-0.3) .. (0,0) .. controls (3.31,0.3) and (6.95,1.4) .. (10.93,3.29)   ;
\draw   (451.98,566.3) .. controls (452.36,561.3) and (456.73,557.55) .. (461.74,557.93) .. controls (466.75,558.32) and (470.5,562.69) .. (470.11,567.69) .. controls (469.73,572.7) and (465.36,576.45) .. (460.35,576.07) .. controls (455.34,575.68) and (451.59,571.31) .. (451.98,566.3) -- cycle ;
\draw   (159.98,806.3) .. controls (160.36,801.3) and (164.73,797.55) .. (169.74,797.93) .. controls (174.75,798.32) and (178.5,802.69) .. (178.11,807.69) .. controls (177.73,812.7) and (173.36,816.45) .. (168.35,816.07) .. controls (163.34,815.68) and (159.59,811.31) .. (159.98,806.3) -- cycle ;
\draw    (170.67,826.51) -- (193.56,788.88) ;
\draw    (145.11,824.44) -- (168,786.82) ;
\draw    (151.53,825.12) -- (174.42,787.49) ;
\draw    (157.6,825.81) -- (180.49,788.19) ;
\draw    (163.67,827.51) -- (186.56,789.88) ;
\draw    (458.87,823.31) -- (481.76,785.68) ;
\draw    (433.31,821.24) -- (456.2,783.62) ;
\draw    (439.73,821.92) -- (462.62,784.29) ;
\draw    (445.8,822.61) -- (468.69,784.99) ;
\draw    (451.87,824.31) -- (474.76,786.68) ;
\draw    (490.87,826.31) -- (513.76,788.68) ;
\draw    (465.31,824.24) -- (488.2,786.62) ;
\draw    (471.73,824.92) -- (494.62,787.29) ;
\draw    (477.8,825.61) -- (500.69,787.99) ;
\draw    (483.87,827.31) -- (506.76,789.68) ;
\draw    (375.62,387.98) -- (460.78,388.05) ;
\draw  [color={rgb, 255:red, 0; green, 0; blue, 0 }  ][line width=0.75] [line join = round][line cap = round] (348.6,245.73) .. controls (342.2,236.8) and (329.89,231.31) .. (322.46,228.83) .. controls (318.37,227.46) and (308.53,223.85) .. (306.88,224.79) ;

\draw (-18,66.4) node [anchor=north west][inner sep=0.75pt]    {$8$};
\draw (-19,218.4) node [anchor=north west][inner sep=0.75pt]    {$7$};
\draw (-14,396.4) node [anchor=north west][inner sep=0.75pt]    {$6$};
\draw (-17,558.4) node [anchor=north west][inner sep=0.75pt]    {$5$};
\draw (-17,700.4) node [anchor=north west][inner sep=0.75pt]    {$4$};
\draw (-15,811.2) node [anchor=north west][inner sep=0.75pt]    {$2$};
\draw (24,85.72) node [anchor=north west][inner sep=0.75pt]    {$xy,\ x^{2} +yz,y^{2} +xz$};
\draw (31,249.72) node [anchor=north west][inner sep=0.75pt]    {$x^{2} +yz,xy,xz$};
\draw (489,84.72) node [anchor=north west][inner sep=0.75pt]    {$z^{2} ,x^{2} -yz ,y^{2} -xz$};
\draw (227,103.72) node [anchor=north west][inner sep=0.75pt]    {$x^{2} +\lambda yz,y^{2} +\lambda xz,z^{2} +\lambda xy$};
\draw (243,124.72) node [anchor=north west][inner sep=0.75pt]    {$\lambda \neq -1,-\omega,-\omega^{2} ,0,2,2\omega,2\omega^{2}$};
\draw (374,61.72) node [anchor=north west][inner sep=0.75pt]    {${\sf j}( \lambda )$};
\draw (263,249.72) node [anchor=north west][inner sep=0.75pt]    {$x^{2} +yz,xy,z^{2}$};
\draw (506,249.72) node [anchor=north west][inner sep=0.75pt]    {$x^{2} +yz,y^{2} ,z^{2}$};
\draw (46,432.72) node [anchor=north west][inner sep=0.75pt]    {$xy,yz,xz$};
\draw (205,430.72) node [anchor=north west][inner sep=0.75pt]    {$x^{2} ,( y+z) z,xy$};
\draw (361,431.72) node [anchor=north west][inner sep=0.75pt]    {$xz,x^{2} +z^{2} ,yz$};
\draw (528,432.72) node [anchor=north west][inner sep=0.75pt]    {$x^{2} ,y^{2} ,z^{2}$};
\draw (123,611.72) node [anchor=north west][inner sep=0.75pt]    {$xy,xz,z^{2}$};
\draw (420,612.72) node [anchor=north west][inner sep=0.75pt]    {$y^{2} ,z^{2},xy$};
\draw (260,716.72) node [anchor=north west][inner sep=0.75pt]    {$x^{2} +yz,xy,y^{2}$};
\draw (417,837.4) node [anchor=north west][inner sep=0.75pt]    {$x^{2} ,y^{2} ,( x+y)^{2}$};
\draw (126,835.37) node [anchor=north west][inner sep=0.75pt]    {$x^{2} ,xy,xz$};
\draw (135,200.72) node [anchor=north west][inner sep=0.75pt]    {\it 2};
\draw (300,388.72) node [anchor=north west][inner sep=0.75pt]     {\it 2};
\draw (470,373.72) node [anchor=north west][inner sep=0.75pt]      {\it 2};
\draw (516,557.72) node [anchor=north west][inner sep=0.75pt]      {\it 2};
\draw (114,370.72) node [anchor=north west][inner sep=0.75pt]    {\it 3};
\draw (219,557.72) node [anchor=north west][inner sep=0.75pt]     {\it 3};
\draw (369,690.72) node [anchor=north west][inner sep=0.75pt]     {\it 3};
\draw (533,795.72) node [anchor=north west][inner sep=0.75pt]    {\it 3};
\draw (140,45.72) node [anchor=north west][inner sep=0.75pt]     {\it 1};
\draw (360,218.72) node [anchor=north west][inner sep=0.75pt]    {\it 1};
\draw (379,25.72) node [anchor=north west][inner sep=0.75pt]    {\it 0};
\draw (591,367.72) node [anchor=north west][inner sep=0.75pt]   {\it 0};
\draw (600,200.72) node [anchor=north west][inner sep=0.75pt]   {\it 0};
\draw (603,48.72) node [anchor=north west][inner sep=0.75pt]    {\it 0};
\draw (214,800.72) node [anchor=north west][inner sep=0.75pt]    {$\infty $};

 \thispagestyle{empty}
\end{tikzpicture}
\caption{Specializations of nets of conics and scheme length.}\label{1fig}
\end{figure}
\vskip 0.2cm
\subsubsection{Hilbert functions $H(R/(V))$ for $ V$ a net of conics in $\mathbb P^2$.}
The Lemma \ref{regslem} below is a consequence of the G. Gotzmann theorems on the Hilbert scheme - see \cite{Gotz} and the exposition in \cite[Appendix C]{IK}. We explain the concept briefly.
\begin{lemma}\label{Maclem}[Macaulay inequality] Recall that $ R_j$ is the vector space of degree-$j$ forms in $R={\sf k}[x_1,\ldots,x_r]$.  Let $d\le \dim_{\sf k}R_j$, and denote by $V_0(r,j,d)$ the span of the first $d$ monomials of $R_j$ in lexicographic order. Set $f(r,j,d)=\dim_{\sf k}R_1V_0(r,j,d)$. Then for any vector subspace $V\subset R_j$ with $\dim V=d$, we have
\begin{equation}\label{Macaulayeq} \dim_{\sf k}R_1V\ge f(r,j,d).
\end{equation}
\end{lemma}
 To state the Gotzmann theorem, in the form that we will use to delimit the possible Hilbert functions of $R/(V)$ and of related algebras, we need some definitions.
Recall that we denote by $\overline{a}$ the infinite sequence $(a,a,\ldots )$ of $a's$. \par
The \emph{Castelnuovo-Mumford regularity degree} of a finitely-generated $R$-module $M$ is the smallest degree $m$ such the $M$ is $m-$regular; $M$ is $m$-regular if  $\Ext^j(M,R)_n=0$ for all $j$ and all $n\le -m-j-1$ (\cite[Proposition 20.16]{Ei}). These definitions become simpler when $A$ is Artinian or has dimension one.

 \begin{definition}\label{regdef} The \emph{regularity degree} of an Artinan algebra $A$ is the socle degree of $A$. A graded algebra $A$ has dimension one when $H(A)=(1,\ldots ,s_u,s,\overline{s})$: here, the regularity degree of $A$ is the smallest degree $u$ for which $H(A)_u=s$.  More generally, \par The \emph{Hilbert polynomial} $\mathfrak h_A(t)$ of an Artinian algebra is its length $|A|$; that of a dimension one algebra is $s$ above.\par
 The \emph{Gotzmann regularity} degree of a polynomial $h(t)$ is the smallest degree $\mathfrak r(h)$ such that all algebras of Hilbert polynomial $h(t)$ are regular in that degree. 
 \end{definition}
 We now define the ancestor ideal $\overline{V}$ of a vector space of forms $V\subset R_j$. {Recall that we denote by $(V)$ the ideal of $R$ generated by $V$.
 \begin{definition}\label{ancestordef}\cite{I3} Let $R={\sf k}[x_1,\ldots,x_r]$ and $V\subset R_j$ be a subvector space. Then the \emph{ancestor ideal} $\overline{V}$ of $V$ in $R$ satisfies
\begin{align}\label{anc1eq}
\overline{V}&=\oplus _{i=1}^j \overline{V}_i+(V) \text { where for $1\le i\le j$ }\notag\\
\overline{V}_i&=V:R_{j-i}=\{h\in R_i \mid  R_{j-i}h\subset V\}.
\end{align}
\end{definition}
 The Gotzmann theorem  (\cite{Gotz},\cite[Appendix C]{IK}) states \par
\begin{proposition}\label{Gotzmannprop}
Assume that  $V\subset R_j$ satisfies $\dim_{\sf k}R_1V=f(r,j,d)$. Then \par
(i). Persistence:  $\dim R_2V$ satisfies 
$$\dim R_2V=f(r,j,d_1), \text{ where } d_1=f(r,j,d).$$
\par
(ii). Scheme $\Z_V$: The ancestor ideal $J=\overline{V}$ (Definition \ref{ancestordef}) - which agrees with the ideal $(V)$ in degree at least $j$ - defines a subscheme $\Z_V=\Proj(R/J)\subset \mathbb P^{r-1}$ whose Hilbert polynomial $p_\Z$ is that giving $\dim_{\sf k} R_i/R_{i-j}V$ for $i$ at least the \emph{Gotzmann regularity degree} $\mathfrak r(p_\Z)$ of $p_\Z$.
\end{proposition}
 For the next Lemma see \cite[Proposition~C.32]{IK}. 
\begin{lemma}[The Gotzmann regularity degree of the constant Hilbert polynomial]\label{regslem} Let $R={\sf k}[x_1,\ldots,x_r]$ be a polynomial ring, let $s$ be a positive integer. Assume that the ideal $I\subset R$, satisfies
\begin{equation*}
 H(R/I)_i=H(R/I)_{i+1}=s \text { for some integer } i\ge s. 
 \end{equation*}
 Then $H(R/(I_i))_k=s $ for $k\ge i$, where $(I_i)$ denotes the ideal generated by $I_i$.\par
 Equivalently, the Gotzmann regularity degree of the 
 constant Hilbert polynomial $p_\Z=s$ is $s$.
\end{lemma}
This result assures that the sequence $H(R/(V))=(1,3,3,2,2,\overline {1})$ cannot occur as the Hilbert function for $A=R/(V)$ for a net $V$ of ternary conics, as $(V)_3$ would be defined by a length 2 scheme so $(V)_{\ge 3}=(J_\Z)_{\ge 3}$ and we would have $H(R/(V))=(1,3,3,\overline{2})$.
\par We next recall the following semicontinuity result.
\begin{lemma}\label{openlem} Let $V(w), w\in \mathcal W$ be a family of constant-dimension $d$ subspaces of $R_j$. Then the condition $\dim_{\sf k}R_1V > k$ is an open condition on the family $V(w)$. Also, the Hilbert function of $H(R/(V))$ is semicontinuous.
\end{lemma}
\begin{proof}The condition $\dim_{\sf k}R_1V > k$ requires that the $rd\times r_{j+1}$ matrix defining the span of $R_1V=\langle x_1 V,\ldots, x_r V\rangle$ in
$R_{j+1}$ has rank greater than $k$ - an open condition. The same can be said for $R_iV(w)$, so  the Hilbert function $H(R/(V(w))), w\in \mathcal W$ is semicontinuous. 
\end{proof}
The following result is well-known.
\begin{lemma}  A complete intersection of generator degrees $(2,2,2)$ has Hilbert function $(1,3,3,1,0)$.
No ideal $I=(V)$ generated by a net of conics can have a different Hilbert function ending in zero.
\end{lemma}
Thus, for example, $H(R/(V))=(1,3,3,2,1,0)$ cannot occur. 
\begin{lemma}[Hilbert functions possible for a net of ternary cubics]\label{elimlem} Let $V$ be a net of conics in $R={\sf k}[x,y,z]$. Then the Hilbert function
$H(R/(V))$ determined by the ideal generated by $V$ must be one of the following:
\begin{equation}\label{HFlisteq}
 (1,3,3,1,0), (1,3,3,\overline{1}), (1,3,3,\overline{2}),(1,3,3,\overline{3}) \text { or } (1,3,3,4,5,\ldots),
\end{equation}
the last only occurring for a net isomorphic to $V=xR_1=\langle x^2,xy,xz\rangle$ (\#2a).
The sequences $(1,3,3,2,\overline{1}),  (1,3,3,3,2,\overline{1}), (1,3,3,3,\overline{2})$ do not occur as Hilbert functions $H(R/(V))$. The sequence ending in $\overline{1}$ occurs only for the $\Pgl(3)$ isomorphism classes \#8a, \#7b in Figure \ref{1fig}; that ending in $\overline{2}$ occurs only for \#7a,6b,6c,5b  in  Figure \ref{1fig}; and that ending in $\overline{3}$
occurs only for the isomorphism classes \#6a,5a,4,2b.
\end{lemma}
\begin{proof} We can eliminate $H(R/(V))= (1,3,3,a, \overline{1}), a>1$, by considering 
 Figure \ref{1fig}
\#7b, the most special with tail $\overline{1}$ whose Hilbert function is $H(R/(V))=(1,3,3,\overline{1})$. Any family specializing to it must have $a=1$, as $\dim R_1V=10-H(R/(V))_3$ can only decrease under specialization, as a consequence of Lemma \ref{openlem}.\par
  We may eliminate $(1,3,3,3,\overline{2})$ by considering \#5b of  Figure \ref{1fig}, $V=(y^2,z^2,xy)$, which is the most special for which  $H(R/(V))$ ends in $ \overline{2}$, so it will have the termwise highest Hilbert function: and here $H(R/(V))_3=2,$ not $3$.\par
  The last statement can be read from the scheme lengths in Figure~\ref{1fig}.
\end{proof}
\vskip 0.2cm\noindent
{\bf Question}.  Do we have all but the last statement of Lemma  \ref{elimlem} over an arbitrary field? The Figure \ref{1fig} requires $\cha {\sf k}\ne 2,3$ and $k$ algebraically closed (see remark before Theorem A1 in \cite{AEI})  but the G.~Gotzmann theorem does not have a restriction on $\cha {\sf k}$, and holds more generally for $B$-flat graded modules over a ring $B[x_1,\ldots, x_r]$ (see \cite[Theorem C.17 and C.29]{IK}).\par\vskip 0.2cm
Finally, we relate the ancestor ideal of a codimension one vector space $V\subset R_j$ to the Gorenstein ideal determined by $V$.
\begin{lemma}\cite[\S 77ff]{Mac},\cite[Lemma 2.17]{IK}\label{ancgorlem} Let $A=R/I$ be an Artinian Gorenstein algebra of socle degree $j$, and let $V=I_j\subset R_j$. Then we have
\begin{equation}
I=\overline{V}+\m^{j+1}.
\end{equation}
\end{lemma}
We will see that this prevents \#2a of Figure \ref{1fig} from being the degree-2 component of a Gorenstein algebra of Hilbert function $T=(1,3^k,1), k\ge 2$ (Lemma \ref{nonemptyGorlem}).
\section{Deformation of algebras in $G_T$ to Gorenstein algebras.}
We recall our Theorem \ref{1thm}:  \par\noindent
\setcounter{thm}{0} 
\begin{thm} Let $T=(1,3^k,1), k\ge 2$. For $ k=2$, the family $G_T$ of graded quotients of $R={\sf k}[x,y,z]$ having Hilbert function $T$, is the Zariski closure of the family $CI_T$ of graded complete intersection quotients of $R$ having Hilbert function $T$; and $G_T$ has dimension $9$.\par When $k\ge 3$ then $G_T$ is in the closure of $\Gor_{\#6a}(T)$, and $G_T$ has dimension eight.  \par In each case $G_T\subset \overline{\Gor(T)}$ and is irreducible.
\end{thm}
\vskip 0.2cm
 We show this below in Theorem \ref{schl3thm} (scheme length of $V=I_2$ is three, Section \ref{3.2sec}), Theorem \ref{k=2GTthm} (case $k=2$, Section \ref{3.4sec}) and Theorem \ref{kge3thm} (case $k\ge 3$, Section \ref{3.6sec}).\vskip 0.2cm
\subsection{The families $\Gor_V(T), \Gor_{\{V\}}(T)$ and $G_V(T),  G_{\{V\}}(T)$.}
\begin{definition}\label{GVdef}
Let $V$ be a net of planar conics. We denote by $\Gor_V(T)$ the family of Artinian Gorenstein
quotients  $A=R/I$ of $R$ having Hilbert function $H(A)=T$, and satisfying $I_2=V$. We denote by $\Gor_{\{V\}}(T)$
 the family of Artinian Gorenstein quotients $A=R/I$ of $R$ for which $I_2 \cong V$ - is equivalent to $V$ under $\Pgl(3)$ isomorphism.  We define $G_V(T)$ and $G_{\{V\}}(T)$ similarly. 
\end{definition} 
\begin{lemma}\label{irredlem} Let $\sf k$ be algebraically closed. Let $T=(1,3^k,1), k\ge 2,$ and $V$ be a net of conics.  Then  $\Gor_V(T), \Gor_{\{V\}}(T), G_{V}(T)$ and $ G_{\{V\}}(T)$ are irreducible varieties.
\end{lemma} 

\begin{proof} The irreducibility of $G_{V}(T)$ and $ G_{\{V\}}(T)$ is shown in Theorem~\ref{dimthm}, page \pageref{dimthm} below: the result depends on the irreducibility of each of the families of nets in Figure \ref{1fig} and as well a review of the fiber of $G_V(T)$ over $V$. The irreducibility of $\Gor_V(T)$ follows from considering that, when non-empty (by Theorem \ref{schl3thm}ii $\Gor_V(T)$ is empty in the case of $V=$\#2b) an element of $\Gor_V(T)$ is determined by a general enough form $F$ in the vector space $ ((V)_j)^\perp$. Such $F$ are parametrized by a dense open of the projective space $\mathbb P( ((V)_j)^\perp)$.
 Alternatively, we could use the classification of Gorenstein algebras $\Gor_V(T)$ in Theorem \ref{5thm}. The variety $\Gor_{\{V\}}(T)$ is the orbit of the irreducible $\Gor_V(T)$ under the action of the irreducible projective group $\Pgl(3)$: it is the image under a morphism of the irreducible variety $\Gor_V(T)\times \Pgl(3)$, so is also irreducible.\vskip 0.2cm\par\noindent
\end{proof}

 Recall that the nets $V$ of scheme length 3 are \#6a,\#5a, \#4,and \#2b, from Figure \ref{1fig};  in
\begin{figure}[h]
\small
\begin{equation*}
\begin{array}{|c||c|c|c|c|}
\#&6a&5a&4&2b\\
\hline
V&xy,xz, yz&xy,xz,z^2&x^2+t^2yz,xy,y^2&x^2,y^2,(x+y)^2\\
((V)_j)^\perp&X^j,Y^j,Z^j&X^j,Y^j,Y^{j-1}Z&W_{\#4}&XZ^{j-1},YZ^{j-1},Z^j\\
\hline
\end{array}
\end{equation*}
\caption{Nets of scheme length three and their duals.}\label{netssl3}
\end{figure}
\normalsize
Figure \ref{netssl3} we give their dual vector spaces $((V)_3)^\perp$. Here $W_{\#4}= \langle ((tX)^2-YZ)Z^{j-2},XZ^{j-1},Z^j\rangle$. A check shows that a generic $F$ in each dual space except \#2b, satisfies $\dim_{\sf k} R_1\circ F=3$. We have in general for a vector space $K\subset R_{j-1}$ that $R_1\circ \langle R_1K\rangle^\perp\subset K^\perp\subset S_{j-1}$: thus we have for these $V$ not in \#2b that $R_1\circ F=((V)_{j-1})^\vee$. It is easy to see then that 
$\dim_{\sf k}R_u\circ F=3$ for $1\le u\le j-2$, hence that $A_F=R/\Ann F$ is Gorenstein of Hilbert function $T=(1,3^k,1)$.\par  \noindent
\begin{lemma}\label{nonemptyGorlem}
The family $\Gor_V(T)$ is nonempty exactly when either\par
a.  $\ell(V)=3, k\ge 2$ and $V$ is equivalent under $\Pgl(3)$ to \#6a,5a, or \#4 - so $V$ is not \#2b; OR\par
b. $ \ell(V)=0$ and $ k=2$. So $V$ is equivalent to \#8b,8c,7c, or \#6d.
\end{lemma}
\begin{proof}
 When $V$ has scheme-length $\ell(V)=3$, the duals 
 $(V_j)^\perp$ have dimension three for $j\ge 3$ and are given in Figure \ref{netssl3}: a generic $F\in (V_j)^\perp$, evidently satisfies $\dim_{\sf k}R_1\circ F=3$ except for $V=$\#2b where $\dim_{\sf k}R_1\circ F=2$. In the former case, evidently $\dim_{\sf k}R_2\circ F=3$, and Lemma \ref{elimlem} implies that $H(R/\Ann F)=T$.\par When the scheme length of $V$ is zero, that is for each $V$ congruent to \#8b,8c,7d,6d  the algebra $R/(V)$ is a complete intersection.\par
 When the scheme length of $V$ is one, it is a specialization of \#8a, for which $(V_j)^\perp=Z^j$ for $j\ge 3$, and $\dim_{\sf k}R_1\circ Z^j=1$, so $\Gor_V(T)$ is empty.  When the scheme length of $V$ is two, it is a specialization of a family isomorphic to \#7a of Figure \ref{1fig}: but for $V=$\#7a $(V_j)^\perp=\langle Y^j,Z^j\rangle$ for $j\ge 3$: again, for $F$ generic in $(V_j)^\perp$, $\dim_{\sf k}R_1\circ F=2$, so $\Gor_V(T)$ is empty. By theorem \ref{2thm} $\Gor_V(T)$ is empty for all specializatioins $V$ of \#7a. \par
 When $V=$\#2a, $\langle  xR_1\rangle$ then by Lemma \ref{ancgorlem} $A=R/I$ with $I\supset x=V:R_1,$ contradicting $H(A)_1=3$; so $\Gor_{\#2a}(T)=\emptyset$.
 \end{proof}
\subsection{Closures of strata $G_{V^\prime}(T)$.}\label{3.2sec}

We here show\vskip 0.2cm\par\noindent
\begin{thm} Assume that the net $V\in \Grass(R_2,3)$ is in the Zariski closure of nets isomorphic to $V^\prime$, and that the scheme length $\ell(V)=\ell(V^\prime)$ and neither $V$ nor $V^\prime$ is \#2a. 
Assume that the net $V\in \Grass(R_2,3)$ is in the Zariski closure of nets isomorphic to $V^\prime$, and that the scheme length $\ell(V)=\ell(V^\prime)$ and neither $V$ nor $V^\prime$ is \#2a. Let $A=R/I$ be an algebra of Hilbert function $T=(1,3^k,1)$ such that $I_2\cong V$. Then $A$ is in the closure of a family $A(x)=R/I(x)$ of algebras having $I_2(x)\cong V^\prime$, that is $A=\lim_{x\to x_0} A(x), x\in X\backslash x_0$. We may take $X$ to be a curve.
\end{thm}
 We show Theorem \ref{2thm} in Theorem \ref{schl3thm} (case scheme length of $V$ is three - including the case $k\ge 3$, Section \ref{3.2sec}, and in Theorem~\ref{scheme2deformthm} (case $k=2$, Section \ref{3.4sec}.\par
Our proof of the following result uses the classification of nets of conics of \cite{AEI}.
\begin{theorem}\label{schl3thm}
Let $T=(1,3^k,1), k\ge 2$ and fix a net of conics $V$ of scheme length three, and assume the algebra $A=R/I$ satisfies $A\in G_T$, and $A=R/I$ where $I_2=V$.\par
(i). Assume further $V$ is not isomorphic to the class \#2b of Figure~\ref{1fig}. Then $A\in \overline{\Gor_V(T)}.$ \par
(ii). Assume $V$ is isomorphic to \#2b of Figure \ref{1fig}. Then $\Gor_V(T)$ is empty but $A\in \overline{Gor_W(T)}$ where $W$ is the class of \#4 in Figure \ref{1fig}.
\end{theorem}
\begin{proof}  Let the algebra $A=R/I$ of Hilbert function $T$ satisfy $I_2=V$ of scheme length 3.
 Then 
\begin{equation}\label{schl3eqn} I_u=(V)_u \text{ for }  2\le u\le j-1
\end{equation}
as $(V)_u\subset I_u $  and $\cod I_u=\cod (V)_u=3 $ in $ R_u$.
Also, $I_j=\Ann\, G$ for some $G\in ((V)_j)^\perp$.  
\par\noindent
{\it Proof of (i).} Lemma \ref{nonemptyGorlem} shows that for $V$ of scheme length three, for $V$ not congruent to \#2b, for a generic $F\in V_j^\perp$, $H(R/\Ann F)=T$. Now suppose $A=R/I\in G_V(T)$ for such $V$. We have $I_j=\Ann\, G$ for some $G\in ((V)_j)^\perp$: since $G=\lim_{t\to t_0}F(t)$ where $F(t)$ is generic for $t\not=t_0$, we have  $A_G=R/I_G=\lim_{t\to t_0}A_{F(t)}$, the limit of elements in $\Gor_V(T)$.  \par\noindent
{\it Proof of (ii).} Assume that $V$ is isomorphic to \#2b in Figure \ref{1fig}. We showed in Lemma \ref{nonemptyGorlem} that $\Gor_V(T)$ is empty.  Let $A=R/I\in G_{\{V\}}(T)$: we have that  $I_j=\Ann G, G=\lim_{t\to t_0}F(t)$ where $F(t)\in
(W_j)^\perp$ for $t\not=t_0$. Here we are using the specialization from Figure \ref{1fig}: we let $t\to 0$ in $W_{\#4}= \langle ((tX)^2-YZ)Z^{j-2},XZ^{j-1},Z^j\rangle$ of Figure \ref{netssl3} above. And we note that $A_{F(t)}\in \Gor_W(T)$ for $t\not=0$. This completes the proof of (ii).
\end{proof}

	\subsection{Closures of strata $G_{V^\prime}(T)$ when $T=(1,3,3,1)$.}
Fix a net  $V$ of conics of $R={\sf k}[x,y,z]$, fix $T=(1,3,3,1)$, and recall that $\ell(V)$ is its scheme length of $V$ as in Equation \ref{leqn}. Let $A\in \Gor_V(T)$. When $\ell(V)=0$ or $\ell(V)=1$ then
$A=R/(V,\m^4)$.                                                                                                                                                                                                                                                                                                                                                          When $\ell(V)=2$, we have by Lemma~\ref{elimlem} that $(H(R/(V)))_3=2$, and since $T_3=1$there exists a non zero $f\in R_3/R_1V$ such that $I=(V,f,\m^4)$. When $\ell(V)=3$, again by Lemma~\ref{elimlem}, we have $(H(R/(V)))_3=3$ and since $T_3=1$ there exist non zero polynomials $f,g\in R_3/R_1V$ such that $I=(V,f,g,\m^4)$. The following Theorem shows that the specialization of such algebras follows the specialization pattern of their corresponding nets, as in the arrows of Figure \ref{1fig}. \par We take two approaches here to parametrization. Often $w$ is an informal parameter in the field $\sf k$. Or, more precisely, we may take $w$ to be a new variable and  write $A(w)=R[w]/I(w)$ for an element of the family $A_{\mathcal W}\to \mathcal W$, with $w\in \mathcal W\cong\mathbb A^1$. Note that the hypothesis of the next Theorem rules out $V \cong (x^2,xy,xz)$ (\#2a) - for that case see
Proposition \ref{2aprop}.
\begin{theorem}\label{scheme2deformthm}
	Let  $T=(1,3,3,1)$ and let $V_w, w\in \mathcal W$ be a family of nets of conics of \emph{finite} scheme length $\ell$ and let $V=V_{w_o}$ be a specific net in the family. Assume further that for $w\ne w_o,  V_w$ is isomorphic to a fixed net $V^\prime$. 
	  Now let the algebra $A(w_o)=R/I(w_o)\in G_T$ satisfy $I(w_o)_2=V_{w_o}$.  Then $A(w_o)$ is a specialization of a family $A(w)=R/I(w)$ of Artinian algebras, also of Hilbert function $T$, for which $I(w)_2$ for $w\ne w_0$ is isomorphic to $V^\prime$.
\end{theorem}
\begin{proof} 
	
	For $\ell(V)=0$, then $A=R/(V)$ so specializing the algebra $A(w)$ is the same as specializing the corresponding net $V_w=I_2$. \par 
	For $\ell(V)=1$, $A=R/(I_2,\m^4)$ is simply obtained by quotienting $R/(I_2)$ of Hilbert function $(1,3,3,\overline{1})$ by $\m^4$. Hence, again, specializing such algebras $A(w)$ is the same as specializing the corresponding nets of conics.\par 
	Suppose that $\ell(V)=2$. We have $A(0)=R/(V,f,\m^4)$ with $f\notin (V)_3$, and let $A(w), w\in \mathcal W$ be the family of algebras over $V(w), w\not=0, V(w)\cong V'$ fixed.   The quotient vector space $R_3/R_1V_w\cong {\sf k}^2$, so $\{R_3/R_1V_w,w\in \mathcal W\}$ is an affine $\mathbb A^2$ bundle over $\mathcal W$. Given a non-zero $f\in R_3/R_1V_{w_o}$ there is a Zariski open subset $\mathcal U\subset \mathcal W$ parametrizing nets such that for $w\in \mathcal U$ the same $f\ne R_1V_w$. Thus $(V_w,f,\m^4)$ for $w\in \mathcal U$ is an ideal defining an Artinian algebra $A(w)$ of constant Hilbert function $T$, whose limit as $w\to w_0$ is $A=A(0)$. \par
	A similar proof holds for $\ell(V)=3$, where $f$ above is replaced by $(f,g)$ and $R_3/R_1V_w, w\in \mathcal W$ is an affine $\mathbb A^3$ bundle over $\mathcal W$. An alternative proof for $\ell(V)=3$ is in Theorem \ref{schl3thm}. \end{proof}

Constancy of the dimension $\ell(V)$ of the space $R_3/R_1V_ w, w\in \mathcal W$ is  key to this proof.\par
The following Example shows that when $\ell=2$ in Theorem~\ref{scheme2deformthm} the deformed algebras we find may not be complete intersections - or even Gorenstein: by Lemma \ref{nonemptyGorlem} $V^\prime $ of scheme length two cannot satisfy $V=I_2$ for a Gorenstein algebra $R/I$.
\begin{example}
For any constant $\alpha\in \sf k$, consider the family of ideals $I(w)=(y^2, z^2+wxz, xy, x^2z+\alpha x^3,x^4)$. For any $w\in \sf k$ the net of conics $I(w)_2=\langle y^2, z^2+wxz, xy\rangle$ has scheme length $\ell(I(w)_2)=2$. The algebras $A(w)=R/I(w)$ have Hilbert function $T=(1,3,3,1)$. We have $I(0)=(y^2, z^2, xy, x^2z+\alpha x^3,x^4)$ and $I(0)_2=\langle y^2, z^2, xy\rangle$, the net of conics \#5b of Figure \ref{1fig}. One can verify that for any $w\neq 0$, $I(w)_2=\langle y^2, z(z+wx), xy\rangle$ is isomorphic to the net of conics \#6b of Figure \ref{1fig}. Indeed, setting $x^{\prime}=y$, $y^{\prime}=wx$ and $z^{\prime}=z$ we get  $I(w)_2= \langle x^{\prime 2}, z^{\prime}(z^{\prime}+y^{\prime}), \frac{1}{w}x^{\prime}y^{\prime}\rangle$ that is the net of conics \#6b of Figure \ref{1fig}.
\end{example}
\begin{remark}\label{fieldrem}  These results are valid over an algebraically closed field $\sf k$ of characteristic zero or characteristic $p>3$ - as then the classification of nets in Figure \ref{1fig} is valid (\cite[Appendix~B, p. 80]{AEI} or \cite{Onda}). They are of course also valid over any field for which the classification of isomorphism classes of nets in Figure \ref{1fig} is valid, but in general the classification of nets over non-algebraically-closed fields is more complicated (see \cite[Appendix B]{AEI}).
\end{remark}
\subsection{The case $T=(1,3,3,1)$: $ G_T$ is in the closure of $CI_T$.}\label{3.4sec}
The deformation of algebras of Hilbert function $T=(1,3,3,1)$ to Gorenstein algebras are treated in 5 different cases, depending on the length of the algebra. For the cases scheme length $\ell(V)=0,1$, the conclusion is easy. For the other cases scheme length $\ell(V)=2,3$ or $\infty$ we give explicit deformations. The proof of the following main result (the first half of Theorem \ref{1thm}), is given below and in the following subsections.  

\begin{theorem}\label{k=2GTthm}  Let $T=(1,3,3,1)$. The family $G_T$ of graded quotients of $R={\sf k}[x,y,z]$ having Hilbert function $T$, is the Zariski closure of the family $CI_T$ of graded complete intersection quotients of $R$ having Hilbert function $T$; and $G_T$ is irreducible of dimension $9$.
\end{theorem}

\begin{proof}[Proof outline]
	We have $I_2=V$ where $\ell(V)\in\{0,1,2,3,\infty\}$. Scheme length $\ell(V)=0$ implies that $R/(V)$ is a complete intersection (CI), of Hilbert function $H(R/(V))=(1,3,3,1,0)$, hence that the algebra $A$ is a CI and there is nothing to show. These algebras are classified up to isomorphism in \cite[Table 1]{AEI}. These are our Figure~\ref{1fig} cases \#8b (continuous family of isomorphisms, parametrized by a function of $\lambda$),\#8c,\#7c, and \#6d. \par
	We suppose next that the scheme length $\ell(V)=1$. These are the cases \#8a or \#7b of Figure \ref{1fig}. We have that the nets of \#8a are in the closure of \#8b since the latter parametrize a dense open in $\Grass(R_2,3)$ that defines CI algebras $A=R/(V);$ thus schemes with $I_2\in $ \#8a are in the closure of $CI_T$. Hence the algebras $A$ with $I_2$ a net of \#7b are also in the closure of $CI_T$ by Theorem \ref{scheme2deformthm}.\par
	For the cases $\ell(V)=2$ or $3$, again by Theorem \ref{scheme2deformthm} it is enough to prove that the classes \#7a and \#6a are in the closure of $CI_T$. We show this by explicit deformations of $A$ to a complete intersection (Sections~\ref{def2sec} and \ref{defsch3sec}). The case $\ell(V)=3$ is also shown in Theorem
	\ref{schl3thm}.
	\par 
	When $\ell(V), V=I_2$ is infinite, the scheme $\Z(V)$ (Proposition~\ref{Gotzmannprop}) may be chosen to be the line $x=0$. This corresponds to the case \#2a which is the special case of Section~\ref{2aasec}, Proposition~\ref{2aprop} for $k=2$. 	\par
	Finally, since $CI_T$ is irreducible, so is $G_T$.
	  
\end{proof}

\subsection{Explicit deformations to AG algebras, $T=(1,3,3,1)$.}

By Theorem \ref{scheme2deformthm} the only explicit deformations needed to prove our main result that algebras $A\in G_T$ are in the closure of $\Gor(T)$, when the scheme length $\ell(V), V=I_2$ is two or three, are those for the cases \#7a and \#6a. We nevertheless give here explicit deformations to Gorenstein algebras for all $A\in G_T$ such that the net $V=I_2$ has scheme length 2 or 3. We explain the cases \#7a and \#6a in detail. 

\subsubsection{Scheme length of $V=I_2$ is two, $T=(1,3,3,1)$.}\label{def2sec} We consider first the nets $V=I_2$ of scheme length is equal to 2, i.e. the cases \#7a, \#6b, \#6c and \#5c. Any algebra $A=R/I$ in $G_T$ having $V=I_2$ satisfies 
\begin{equation}\label{feq}
I=(V,f,\mathrm{m}^4) \text{ where }f\in R_3\quad f\notin (V).
\end{equation}
 We give below the deformations needed on $I$ so that the deformed algebra $A(t), t\ne 0$ is a complete intersection (CI).
\begin{enumerate}
	\item \textbf{Case \#7a.} Let the Artinian algebra $A=R/I$ of Hilbert function $H(A)=(1,3,3,1)$ satisfy $I_2=V=(zy+x^2,zx, yx)$, \#7a in Figure \ref{1fig}. Here $V$ is symmetric in $y,z$ and $R_3/R_1V\cong \langle z^3,y^3\rangle$. We take $f=z^3+ay^3$ with $a\ne 0$. \par
	(When $f=y^3$ or $z^3$ we argue that the algebra $A$ deforms first to a case where $f=z^3+ay^3, a\ne 0$.) \par
	{\it Claim}: $A$ is in the closure of the complete intersections of the form $A=R/I(t))$ such that $V(t)=I(t)_2$ is in \#8b for $t\ne 0$.\par
	\begin{proof}[Proof of Claim] (i).
		Consider $f_1=zy+x^2,$ $f_2=zx+t_1y^2,$ $f_3=t_2z^2+xy$, and let $I(t)=(f_1,f_2,f_3,f)$, where $t$ denotes $(t_1,t_2)$. \par
		We assume henceforth that 
		\begin{equation}\label{deformeq} \text{ $at_1t_2\neq 0$ and $t_1t_2\neq -1$}. 
		\end{equation}
		Then $x^4\in I(t)$, since for $at_1t_2\neq 0$ and $t_1t_2\neq -1$, 
		$$(t_1t_2+1)x^4=((t_1t_2+1)^2-t_1t_2yz)f_1+t_2z^2f_2-xzf_3.$$
		{Also,} since $(1+t_1t_2)x^2y=xf_3-t_2zf_2+t_1t_2yf_1$, then $$R_1x^2y=(x^3y,x^2y^2,x^2yz)\subset I(t).$$ Similarly, since $(1+t_1t_2)xy^2=yf_3+t_2xf_2-t_2zf      _1$, then $$R_1xy^2=(x^2y^2,xy^3,xy^2x)\subset I(t);$$
		and since $(1+t_1t_2)yz^2=t_1yf_3-xf_2+zf_1$, then $$R_1yz^2=(xyz^2,y^2z^2, yz^3)\subset I(t)\quad \text{for $t\not=0$}.$$ Also,\par
	$x^3z=x^2f_2-t_1x^2y^2, \, t_2xz^3=xzf_3-x^2yz, \,y^3z=y^2f_1-x^2y^2\\
 z^4=zf-ay^3z,\, x^2z^2=xzf_2-t_1xy^2z,$
 and $ay^4=fy-z^3y$.\par Hence, for $a\ne 0$ and $t$ satisfying Equation \eqref{deformeq} (that is, $t_1t_2\ne 0,-1$) the ideal $I(t)$ contains $\m^4$ and $HF(R/I(t))=(1,3,3,1)$.\vskip 0.2cm\par
	(ii.)	We now want to find a relation between $t_1$ and $t_2$ so that $I(t)=(f_1,f_2,f_3)$. A calculation by Macaulay2 shows that $$t_1t_2f-t_1zf_3+(t_1+at_2)xf_1-at_2yf_2=(t_1+at_2)x^3$$
		Therefore, $f\in (f_1,f_2,f_3)$ if and only if $t_1=-at_2
		\ne 0$. Write $t_2=t$ and $t_1=-at$ and then we have that
		\begin{equation}\label{7aCIeq}
			I(t)=(zy+x^2,zx-aty^2, tz^2+xy)
		\end{equation} is a complete intersection, as desired, provided
		$t_1=-at,t_2=t\not=0$, and $t_1t_2=-at^2\not=-1$. This is  Figure \ref{1fig}, entry \#8b.
	\end{proof}
	\item \textbf{Case \#6b.} Let $V=(x^2,xy,z(z+y))$. Since $R_3/R_1V\cong\langle y^2z,y^3\rangle$ we take $f=zy^2+ay^3$ in Equation \eqref{feq}; and consider the deformed algebra $A(t)=R(t)/I(t)$ where $$I(t)=(z^2+yz,at^2yz+x^2,aty^2+xy).$$
	Then $A(t)$ is a CI, provided that $at(at-1)\neq 0$. 
	\item \textbf{Case \#6c.} $V=(z^2+x^2,yz,xz)$ and $f=y^3+axy^2$. The deformation $A(t)=R(t)/I(t)$ with $$I(t)=((1+at)x^2+t(1+at)xy+t^2y^2+z^2,yz+tx^2,xz),$$
	is a CI if $at(1+at)\neq 0 $. 
	\item \textbf{Case \#5b.} $V=(z^2,y^2,xy)$ and $f=x^2z+ax^3$. Then, with $at\neq 0$ we have that $A(t)=R(t)/I(t)$ is a CI for $$I(t)=(z^2,y^2+at^2xz,xy+atx^2,x^2z+ax^3).$$

\end{enumerate}
\subsubsection{Scheme length of $V=I_2$ is three, $T=(1,3,3,1)$.}\label{defsch3sec}
The cases $\#6a$, $\#5a$, $\#4$ and $\#2b$ of Figure \ref{1fig} correspond to nets $V$ having scheme length 3. An algebra $A=R/I$ in $G_T$ having $V=I_2$, with $\ell(V)=3$ satisfies 
\begin{equation}\label{fgeq}
I=(V,f,g,\m^4) \text{ where $f,g\in R_3$ and $\langle f,g\rangle \cap R_1V=0$}.
\end{equation}
Here  $\langle f,g\rangle $ denotes the vector space span of $f,g$.
\begin{enumerate}
	\item \textbf{Case \#6a,} Let $V=(xy,yz,xz)$ of scheme length 3. The ideal $I\in G_{\{V\}}(T)$ with $T=(1,3,3,1)$ is up to isomorphism of the form \eqref{fgeq} where $f=z^3+ax^3$ and $g=y^3+bx^3$, with $ab\neq 0$. Consider the following deformation, for {\bf t}$=(t_1,t_2,t_3)$,
	
	\quad $I(\textbf{t})=(f_1,f_2,f_3,f,g)$ where $f_1=yz+t_1x^2$, $f_2=xz+t_2y^2$, $f_3=xy+t_3z^2$.  A computation by Macaulay2 shows that the Hilbert function of $A/I(\textbf{t})$ is $(1,3,3,1)$ with the conditions $abt_1t_2t_3\neq 0$ and $1+t_1t_2t_3\ne 0$. We want to find a relation between the parameters so that $f$ and $g$ are in $(f_1,f_2,f_3)$. Using Macaulay2 again we get $$(t_1+at_3)x^3=xf_1-zf_3+t_3f$$ and $$(bt_2-at_3)x^3=-yf_2+zf_3-t_3f+t_2g.$$
	Assuming furthermore that $t_1+at_3=0$ and $bt_2-at_3=0$, and with a change of variable $t_1=-abt$, $t_2=at$ and $t_3=bt$, we set $$I(t)=(yz-abtx^2,xz+aty^2,xy+btz^2)\supset(f,g)$$ with $abt\neq 0$ and $1-a^2b^2t^3\ne 0$. Then $A(t)=R(t)/I(t)$ is a subfamily of CI's in the isomorphism classes \#8b.
	\item \textbf{Case \#5a.} $V=(z^2,xz,xy)$. We can take - possibly after an initial deformation - $f=yz^2+ax^3$ and $g=yz^2+by^3$. The deformed algebra $A(t)=R/I(t)$ where for $abt\neq 0$, $$I(t)=(z^2-a^2tx^2,xz+aty^2-btyz,xy),$$
	defining a complete intersection  $A(t)=R(t)/I(t)$.
	\item \textbf{Case \#4.} $V=(x^2+yz,xy,y^2),$ $f=xz^2+az^3$ and $g=yz^2+bz^3$. The desired deformation is, for {\bf t}$=(t_1,t_2,t_3)$, $$I((\textbf{t})=(x^2+yz+t_1z^2, xy+t_2z^2,t_3xz+y^2+t_4z^2)$$ where $t_1=at$, $t_2=\frac{b}{a^2-b}t(b-at)$, $t_3=bt$ and $t_4=\frac{ab}{a^2-b}t(b-at)$ for $atb(a^2-b)(b-at)\ne 0$ and $t_2+ab\ne 0$.  Then $A(t)=R(t)/I(t)$ is a CI.
	\item \textbf{Case \#2b.} After a change of variables and an initial deformation we have $V=(z^2,yz,y^2)$, $f=zx^2+ax^3$ and $g=yx^2+bx^3$. We consider the following deformation 
	$$I(t)=(z^2+(at^2-bt)zy+atzx+a^2t^2x^2,yz,btxz+y^2)$$
	which defines a complete intersection  $A(t)=R(t)/I(t)$ when $abt\ne 0$.
	\end{enumerate}
\subsection{The case  $T=(1,3^k,1), k\ge 3$: $G_T$ is in the closure of $ {\Gor_{\#6a}(T)}$.}\label{3.6sec}
We begin with a simple observation.
\begin{remark}\label{kge3rem}
Since $H(R/(V))_i=\ell(V)$ for $i\ge 3$, the only $V=I_2$ that are possible for a graded algebra $A$ of Hilbert function $T=(1,3^k,1), k\ge 3$ are\par\noindent
(i). those of scheme length $\ell(V)=3$: $\#6a$, $\#5a$, $\#4$, $\#2b$; OR\par\noindent
(ii). $V$ from \#2a from Figure \ref{1fig} where the scheme $\Z(V)$ is a line. 
\end{remark} 
\begin{theorem}\label{kge3thm} \phantom{AlahambraArdvarksAcclaimAstoundingAxiom} Let $T=(1,3^k,1)$ with $k\ge 3$. Then $G_T$ is in the closure of $\Gor_{\#6a}(T)$, the Artinian Gorenstein algebras $A=R/I$ where $I_2\cong (xy,xz,yz)$, or, in other words those for which $I=\Ann F, F\cong X^j+\alpha_1Y^j+\alpha_2Z^j$; and $G_T$ is irreducible of dimension eight. \par Assume further that $\sf k$ is closed under taking $j$-th roots. Then  $G_T$ is in the closure of the family $A_{\mathcal L}=R/I_{\mathcal L}$ where $I_{\mathcal L}=\Ann\, F$ for $ F=L_1^j+L_2^j+L_3^j$ with ${\mathcal L}=\{L_1,L_2,L_3\}$ linearly independent in $S_1$.

\end{theorem}\noindent
This is the second part of Theorem \ref{1thm}. It follows from Remark~\ref{kge3rem}, Theorem~\ref{schl3thm} above, Proposition \ref{2aprop} below and the irreducibility of 
$\Gor(T)$ for $T$ a Gorenstein sequence of codimension three \cite{Di}.

\subsubsection{Scheme $\Z_V$ a line, $\ell(V)=\infty$.}\label{2aasec}
 The case $V$ in \#2a is the only case where the length of the scheme is infinite, with Hilbert function $(1,3,3,4,5,\dots)$: the scheme $\Z_V$ is the line $x=0$. We next show that the algebra $A=R/I$, with $V=I_2$ in \#2a, deforms to an algebra $A(w)$ where $V(w), w\ne w_o$ is in \#6a. Then Theorem \ref{schl3thm} provides a deformation to a Gorenstein algebra.  Note that the socle degree $j_A=k+1$.\par Complementary to the following Proposition is Corollary \ref{closurecor} (ii), which states that  $G_{\{V\}}(T),$ $ T=(1,3^k,1), k\ge 3$ is not in the closure of any $G_{\{V^\prime\}}(T)$ for which $d(V^\prime)\le 5$.
\begin{proposition}\label{2aprop} Let $T= (1,3^k,1), k\ge 2$. Suppose $A=R/I$ where $V=I_2=\langle x^2,xy,xz\rangle$ (the class \#2a) and $H(A)=T$. Then $A=A(w_o)$ is a specialization of a family  $A(w), w\in W$ where $A(w)=R(w)/I(w), w\ne w_o$ satisfies $V(w)=I(w)_2$ is in class \#6a, for $w\ne w_o$, and $H(A(w))=T$.\par
\end{proposition}
\begin{proof}
 Since $H(R/I_2)=(1,3,3,4,5,\ldots )$, in order to achieve the Hilbert function $H(A)=T$  we need that $I$ has a generator $f\in {\sf k}[y,z]_3$, and further generators $g,h\in {\sf k}[y,z]_{j}$. The ideal $(V,f)$ satisfies $H(R/(V,f))=(1,3,\overline{3})$ (if only because it is essentially generated  in ${\sf k}[y,z]$ by a degree three form), so $(V,f) $ defines a scheme $\Z(w_0)$ of length 3 concentrated on the line $x=0$ of $\mathbb P^2=\Proj R$. 
We have $ I=(V,f,g,h, \m^{j+1} ) $  (we might need a generator for $I$ in degree $j+1$).
Recall that in $\Hilb^n \mathbb P^2$ there is an open dense subfamily parametrizing smooth schemes \cite{Fo,BriGal}. It follows that
there is a parameter variety $W\subset \
\Hilb^3(\mathbb P^2)$ (we can take $W$ to be a curve) such that
the scheme $\Z=\Z(w_0) $ defined by $(V,f)$ deforms to a scheme $\Z(w)\subset \mathbb P^2$, defined by the ideal $(V(w))\subset R$, such that for $w\ne w_0$ the scheme $\Z(w)$ is smooth, comprised of three non-collinear points. Thus, for $w\ne w_0$, we have that $V(w)=I(w)_2$ defines $\Z(w)$, and lies in the $\Pgl(3)$ equivalence class of \#6a (here \#6a: $V=(xy,xz,yz)$ corresponds to the scheme defined by the three points $(1,0,0), (0,1,0)$ and $(0,0,1)$).  Asking that the ideal $(I(w)_2)$ have intersection $0$ with $ \langle g,h\rangle $ is an open condition on $w$, satisfied for $w=w_0$ by $(V,f)$.  \par
Thus, we may deform $A=A(w_0)$ to the algebra $A(w)=R/I(w)$ where $I(w)=(V(w), g,h,\m^4)\, H(A(w))=T$ and $I(w)_2=V(w)$ is in the isomorphism class \#6a of nets of conics  for $w\ne w_0$. 
 \begin{remark}[Counterexample]\label{nondeformrem} We will show that for $V=(xy,xz,x^2)$ (\# 2a) the family $G_{\{V\}}(T)$ of graded algebras $A=R/I$ of Hilbert function $T=(1,3^k,1), k\ge 3$ having $I_2\cong V$, form an irreducible variety of dimension seven, of known structure (Theorem \ref{dimthm}(iii)). We will show also that the analogous family of quotients $A=R/I$ for \#4,\#5b or \#6d, has dimension 6, and for \#5a,\#6b,\#6c dimension 7 (Theorem \ref{dimthm}(i),(ii)). It follows that when $\ell(V)\not=\ell(V^\prime)$ and $V$ is a limit of nets in the class of $V^\prime$, it may occur that not all algebras in $G_{\{V\}}(T)$ are in the closure of $G_{\{V^\prime\}}(T)$.
 \par
Although we cannot make an analogous counterexample involving only $V^\prime, V$ of finite lengths for $T=(1,3^k,1), k\ge 3$, as the only occurring finite length is three, we can readily make an analogous example for $k=2$, relying on Theorem \ref{dimthm} and Corollary \ref{closurecor} (see Remark \ref{closurerem}). Namely, choose any pair $(V^\prime, V)$ from Figure \ref{1fig}, such that $V^\prime$ specializes to $V$, with dimension $$d(V^\prime)=d(V)+1 \text { and } \ell(V)\ge \min\{2, \ell(V^\prime)+1\};$$ then $G_{\{V\}}(T)\nsubseteq \overline{G_{\{V^\prime\}}(T)}$, for dimension reasons. Examples of such pairs are (\#7a,\#6a), (\#6b,\#5a), and  (\#5b,\#4). 

 \end{remark}
\end{proof}
\subsection{Graded algebras $G_T$ in codimension three not in the closure of $\Gor(T)$.}\label{nclosesec}
We have proven in Theorem \ref{1thm} that for  $T=(1,3^k,1), k\ge 1$ the Zariski closure $\overline{\Gor(T)}=G_T$. To show that this result is sharp we determine the smallest-length codimension three Gorenstein sequences $T=(1,3,5,3,1)$, $ (1,3,6,3,1)$ and $(1,3,6,6,3,1)$ such that $G_T$ has two irreducible components, $\Gor(T)$ and another component $U_T$ that we will describe.\vskip 0.2cm\par\noindent
{\it Notation and preliminary facts.}
Below we set $R={\sf k}[x,y,z]$ and $S={\sf k}_{DP}[X,Y,Z]$, over an infinite field $\sf k$. Recall that $R$ acts on $S$ by contraction (Equation \eqref{contracteq}). 
Recall
\begin{lemma}\label{codlem}\cite[Prop. p.67 \S 4.2]{ACGH}. The rank $r$ locus constructed from a generic $m\times n$ matrix $M_{gen}$ satisfies 
\begin{equation}\label{codrankeq}
\text{codimension of rank $r$ locus}=(m-r)(n-r).
\end{equation}
This is also the maximum codimension for the rank $r$ locus of any $m\times n$ matrix $M$.
\end{lemma}
We denote by $r_i=\dim_{\sf k}R_i=\binom{i+2}{2}$; the notation $1_u$ denotes a $1$ in degree $u$.  We recall
\begin{definition}
A \emph{compressed} Gorenstein sequence $H$ of codimension three and socle degree $j$ satisfies 
\begin{equation}\label{compeq}
H=H(j)=(1,3,\ldots, h_u,\ldots, 1_j) \text 
{ where }h_u=\min(r_u,r_{j-u}).
\end{equation}
When $j$ is even these have been called {\it extremal } Gorenstein sequences \cite{Sch}.
We will term the Hilbert function $T$ \emph{almost compressed} when $j$ is even, and we have $$T=H(j)-1_{j/2}.$$ For example, $T=(1,3,6,3,1)$ is compressed, and $T=(1,3,5,3,1)$ is almost compressed.
\end{definition}

\vskip 0.3cm
\begin{figure}[h]
\begin{center}
\small
$\begin{array}{c||ccc|cc|}
&a=6&a=5&a=4&b=6&b=5\\
\hline
\dim \Gor(T)&14&13&11&20&15\\
\dim U_T&\ge 16&\ge14&\ge 10^\ast&\ge 22&\ge 14^\ast\\\hline
&&T=(1,3,a,3,1)&&T=(1,3,b,b,3,1)&\\
\hline
\end{array}$
\caption{Dimensions of multiple components of $\Gor(T)$,}\label{compfig}
 for $ T=(1,3,a,3,1), a=5,6$, and $T=(1,3,6,6,3,1)$.
\end{center}
\end{figure}\vskip -0.7cm
By the asterisk on two entries of Figure \ref{compfig}
 we mean\par
\quad
\small $\ge 10^\ast: \, U_T, T=(1,3,4,3,1)$ is likely in the closure of $\Gor(T)$;\par
\quad \small $\ge 14^\ast: \, U_T, T=(1,3,5,5,3,1)$ may be in the closure of $\Gor(T)$.\par
\normalsize
\begin{definition}\label{UTdef} Given a Hilbert function $T=(1,3,\ldots, 1_j,0)$ we will denote by $U_T\subset G_T$ the family of graded Artinian quotients $ A=R/I$ of $R$ such that $I_j^\perp=L^{[j]}$ the (divided) power of a linear form $L\in S_1$.  We denote by $U_{T,s}, 1\le s\le 3$ the subfamily of $G_T$ for which the socle-degree $j$ dual generator $F_j$ satisfies $\dim_{\F} R_1\circ F=s$. Equivalently $H(R/\Ann F_j)=(1,s,\ldots,s,1_j)$, satisfying $h_1=h_{j-1}=s.$ This is just the subfamily of $G_T$ where $F_j$, after a suitable coordinate change, involves a minimum of exactly $s$ variables.  Here $U_T=U_{T,1}$. 
\end{definition}
\begin{example}[Algebras in $U_{T,i}$]\label{UTex} Let $T=(1,3,3,1)$. Then the Artinian algebra $A(1)=R/I(1), I(1)=\Ann (X^3,Y^2,Z^2)=(xz,yz,xy,y^3,z^3, x^4)$ is in $U_{T,1}$. The algebra $A(2)=R/I(2), I(2)=\Ann (X^2Y,Z^2)=(xz,yz,y^2,x^3,z^3)$ is in $U_{T,2}$; and the algebra $A(3)=R/I(3), I(3)=\Ann(XYZ)=(x^2,y^2,z^2)$ is in $U_{T,3}$.
\end{example}

By {\it closure} of an algebraic set we will mean its Zariski closure.
\begin{lemma}\label{closurelem} For $T=(1,3,\ldots,3_{j-1},1_j,0)$ the closure $\overline{U_T}\subset G_T$ is just
$\overline{U_T}=U_T$.  The closure of $U_{T,2}$ can contain no elements of $U_{T,3}$. When $T=(1,3^k,1)$ we have $U_{T,3}=\Gor(T)$.\par For any Gorenstein sequence $T=(1,r,\ldots,r,1)$ with $r>1$ we have $\overline{U_T}\cap \Gor(T)=\emptyset$.
\end{lemma}
\begin{proof} Let $A(w),w\in W$ be a subfamily of\, $G_T$, and let $F(w)\in S_j$ (up to non-constant multiple) be the degree-$j$ dual generator of $A(w)$. The condition $\dim_\F R_1\circ F(w)\le s$ asks that the module $R_1\circ F\subset S_{j-1}$ have rank no greater than $s$, which defines a closed condition on $W$. This also implies the last statement of the Lemma.\par
When $T=(1,3^k,1)$ and $F\in S_j$ is a dual generator satisfying $\dim R_1\circ F=3$ and $\dim R_i\circ F\le 3$ for $1\le i\le j-1$, we have by symmetry that the Hilbert function of $A_F=R/\Ann F)$ satisfies $H_F=(1,3,\ldots,3,1)$; but the Macaulay equation~\eqref{Macaulayeq} implies that for $i\ge 2$, $$H(R/\Ann F)_i=2\Rightarrow H(R/\Ann F_i)_{i+1}\le 2.$$ Since $(H_F)_{j-1}=3$ this implies that $H_F=(1,3^k,1)$ so $A=R/\Ann F$ and is an element of $U_{T,3}=\Gor(T)$.  Evidently, if $A_F\in \Gor(T)$ then $\dim R_1\circ F=3$: we have shown $U_{T,3}=\Gor(T)$.
\end{proof}
Consider the family $\CI (\mathcal D)$ of graded complete intersection quotients of $R={\sf k}[x,y,z]$ having generator degrees $\mathcal D=\mathcal(d_1,d_2,d_3)$, and denote by $T(\mathcal D)$ their Hilbert function, so $\CI(\mathcal D)=CI_{T(\mathcal D)}$. It is well-known (as \cite[Example 3]{I2a},also \cite[\S 4.6]{Di}, \cite[Equation~3.1.4]{IK}) that with $T=T(\mathcal D)=(t_0,t_1,\ldots )$, 
\begin{equation}\label{CIeqn}
\dim \CI(\mathcal D)=t_{d_1}+t_{d_2}+t_{d_3}.
\end{equation}
\begin{theorem}[Smallest $G_T$ not in $\overline{\Gor(T)}$]\label{small2compprop} Let $5\le a\le 6$ and $T=(1,3,a,3,1)$, 
 or let $T=(1,3,6,6,3,1)$. Then $G_T$ has at least two irreducible components, $\overline{\Gor(T)}$, and a second, larger component $U_T$ of Definition \ref{UTdef}, whose dimensions are given in Figure \ref{compfig}.

\end{theorem}
\par

\begin{proof}  For each of the cases $T=(1,3,a,3,1)$ and $T=(1,3, 6,6,3,1)$ we show the step-by-step construction of $U_T$. \par i. First consider $T=(1,3,a,3,1)$. Choose an element $L\in S_1$: this is a choice parametrized by $\mathbb P(S_1)=\mathbb P^2$. We will take $F=L^4$ as our first
 dual generator, so $A^\vee \supset \langle 1,L,L^2,L^3,L^4\rangle.$\par
i.a. Now let $a=6$. Let $W_L$ parametrize the three-dimensional sub-vector spaces of $S_3$ that contain $\langle L^3\rangle$. Here $W_L\cong \Grass(9,2)$, in effect parametrizing two dimension subspaces of $S_3/\langle L^3\rangle.$\par
Let $U_L\subset G_T$ be the subset comprised of algebra quotients $A_{\mathcal B}$ of $R$ having dual generators $\langle L^4,\mathcal B\rangle$ where $\mathcal B\in W_L$  is general enough so that $R_1\circ \mathcal B=S_2$. That this occurs for $\mathcal B$ in a Zariski open dense subset of $W_L$ is a consequence of  the theory of relatively compressed algebras \cite[Theorem III]{I2}. Thus, each Artinianalgebra $A_{\mathcal B}, B\subset U_L$ satisfies $A_{\mathcal B}=R/\Ann \,\langle\mathcal B,L^4\rangle$, and has Hilbert function $H(A_\mathcal B)=T=(1,3,6,3,1)$.\par
Evidently, $U_T$ of Definition \ref{UTdef} satisfies
 \begin{equation}\label{dim1eq}\pi: U_T\to \mathbb P^2:  \pi(A_{\mathcal B})= L,
 \end{equation}
  is the bundle whose fiber over $L\in S_1$ is $U_L$. We have $U_T\subset G_T$, and $$\dim U_T=\dim \mathbb P^2+\dim \Grass(9,2)=16.$$  But here the algebras $A\in\Gor(T)$ are parametrized by their dual generators $F_A\in S_4\mod {\sf k}^*$-multiple: this is an open dense subvariety of $\mathbb P(S_4)\cong \mathbb P^{14}$ so $\Gor(T)$ has dimension 14, smaller than that of $U_T$.  It follows that $U_T$ and $\overline{Gor(T)}$ are two irreducible components of $G_T$, when $T=(1,3,6,3,1)$.\vskip 0.2cm\par
i.b. Now let $a=5$. Again we fix $L\in S_1$ and take the dual generator $F=L^4$. Then to determine an element of $G_T$, we must consider the elements $\mathcal B$ of $W_L$ (as above, three-dimensional subspaces of $S_3$ containing $L^3$) such that 
\begin{equation}\label{smalleq}
\dim R_1\circ \mathcal B =5.
\end{equation} Recall  that we have $\mathcal B=\langle L^3,v_1,v_2\rangle $, so 
\begin{equation}\label{dim2eq}
R_1\circ \mathcal B=\langle L^2,R_1\circ v_1, R_1\circ v_2\rangle \subset S_3.
\end{equation}
The condition \eqref{smalleq} here asks that an ostensibly $6\times 7$ matrix of Equation \eqref{dim2eq} have rank $5$. The codimension of this condition, by Lemma \ref{codlem} and Equation \eqref{codrankeq}, is no greater than $(7-5)(6-5)= 2$. We let $U_L$ parametrize those elements (3-dimensional vector spaces) $\mathcal B\in W_L$ satisfying the condition \eqref{smalleq}, so $A_{\mathcal B}$ has Hilbert function $H(\mathcal B)=T(a)=(1,3,5,3,1)$. Thus, $U_T$ is a bundle over $ \mathbb P^2$ (parametrizing the $L$) whose fiber over a linear form $L$ is $U_L$.\par
From the above codimension count and Equation~\eqref{dim1eq} we have that $\dim U_T(5)=16-2=14$;  however, $A\in\Gor\, T, T=(1,3,5,3,1)$, is defined by the vanishing of $Cat_{2,2}(F_A)$, the catalecticant of the dual generator $F_A\in S_4$, so $\Gor(T)$ has dimension $14-1=13$.

Thus, when $a=5$, $G_T$ has (at least) the two irreducible components $U_T$ and $\overline{\Gor(T)}$.
\vskip 0.2cm\par
ii. Now consider $T=(1,3,6,6,3,1)$. The dimension of $\Gor(T)=\dim {\mathbb P}(S_5)=20$, as $A\in \Gor(T)$ is parametrized by the dual generator $F_A\in S_5$, up to non-constant scalar multiple.  We next construct $U_T$ similarly to  the previous case (i).\par
As before, choose $L\in S_1$, a choice parametrized by $\mathbb P(S_1)=\mathbb P^2$, and now take $F=L^5$, so $A^\vee\supset \langle 1,L,L^2,L^3,L^4,L^5\rangle$. Let $W_L$ parametrize the three-dimensional sub-vector spaces $\mathcal B\subset S_4$ containing $L^4$. Here 
\begin{equation}\label{WL2eqn} W_L\cong  \Grass(S_4/\langle L^4\rangle,2)=\Grass(14,2) \text{ of dimension } 24.
\end{equation}
 By the theory of relatively compressed algebras, for a dense open set of  $\mathcal B \in W_L$ we have $\dim R_1\circ \mathcal B=7$ \cite[Theorem III]{I2}. That is, $\mathcal B=\langle L^4,v_1,v_2\rangle$ and the 6 partials of general enough elements $v_1,v_2$ are linearly independent of $L^3$. But we need $\dim_{\sf k}R_1\circ \mathcal B =6$. The matrix for $R_1\circ \mathcal B\subset S_3$ is a $7\times 10$ matrix: by Lemma~\ref{codlem} and Equation \eqref{codrankeq} the rank six locus has codimension at most $(7-6)(10-6)=4$. Then $U_T$ is the family of algebras $A_{\mathcal B}=R/\Ann\langle L^5,\mathcal B\rangle$ such that $\dim R_1\circ \mathcal B=6 $ and $R_2\circ \mathcal B=S_2$ (the latter is again true generically for an open dense family of $\mathcal B$ by \cite[Theorem III]{I2}.\par 
  We now will show that in general $R_2\circ \mathcal B= S_2$, even on the locus where $\dim R_1\circ \mathcal B=6$. Suppose $R_2\circ \mathcal B\not=S_2$. Choose $f\in R_2$, a choice of $f\in\mathbb P(R_2)= \mathbb P^5$.  Then $H(R/(f))=(1,3,5,7,9,\ldots)$ and $\mathcal B$ must be a $3$-dimensional space inside a $9$ dimensional space $f^\perp\cap S_4$, or in effect, since $L$ is chosen, a $ 2$-dimensional  space inside an $8$ dimensional space. It follows that such $\mathcal B$ are parametrized by a $\Grass(8,2)$, a $12$ dimensional family: adding the choice of $f$ and of $L$ gives a $19$-dimensional family. But the family $U_L$ of $ \mathcal B\supset L^4$ and satisfying $\dim R_1\circ B=6$ has by Lemma \ref{codlem} and Equation \eqref{WL2eqn} dimension at least
 $$\dim U_T\ge 2+24-4=22,$$ 
 so generically $\mathcal B$ satisfying the rank condition $\dim_{\sf k}R_1\circ \mathcal B=6$  and $R_2\circ \mathcal B=S_2$ has dimension at least $22$, compared to $20$ for $\Gor(T)$ so $G_T$ has (at least) the two irreducible components $U_T$ and $\overline{\Gor(T)}$ when $T=(1,3,6,6,3,1)$.\par
 This completes the proof of the Theorem.
\end{proof}\par

\vskip 0.2cm\noindent
We include an example due to the fourth author.
\begin{example}[J. Yam\'{e}ogo] \cite[Lemma 8.3, p. 251]{IK}.  When $T=(1,3,4,3,1)$ then
\begin{enumerate}[(i).]
\item The ideal $I=\Ann W, W=\langle L^4,X^3,Y^3,XY\rangle$ which satisfies $I=(xz,yz,x^2y,xy^2,x^4,y^4,z^5)$, is a singular point of $G_T$.
\item Let $f\in S_4$ be in the closure $\overline{\Gor(T)_4^\perp}$ of the dual generators of ideals in $\Gor(T)$. Then the fiber of $G_T$ over $f$ (those algebras $A=R/I\in G_T$ such that $I_4=(\Ann f)_4$) consists of those ideals $I $ ``in'' $G_T$ with $I\subset J=\Ann f$.
\end{enumerate}
\begin{question} For $T=(1,3,4,3,1)$ do we have $\overline{\Gor(T)}=G_T$?
\end{question}
The answer ``Yes'' would imply that the cases $T=(1,3,6,3,1)$ and $T=(1,3,5,3,1)$ are the only Gorenstein sequences of socle-degree $4$ and embedding dimension three, where $\overline{Gor(T)}\not=G_T$.
\end{example}
\section{Isomorphism classes of graded Artinian algebras of Hilbert function $T=(1,3^k,1)$.}\label{isomsec}
Recall that $R={\sf k}[x,y,z], S={\sf k}_{DP}[X,Y,Z]$ the action of $R$ on $S$ is contraction. Here $\{V\}$ denotes the $\Pgl(3)$ isomorphism class of nets equivalent to the vector space $V$,  $\Grass_{\{V\}}=\Grass_{\{V\}}(R_2,3)$ parametrizes the family of nets of conics $\Pgl(3)$ isomorphic to $V$, and $G_{\{V\}}(T)$ parametrizes the family of all graded algebras having $I_2\cong V$. The dimension $d(V)$ of the family $\Grass_{\{V\}}$ is indicated at the left of Figure \ref{1fig}.  We will denote by $\Gor_V(T)$ and $G_V(T)$ respectively the family of 
Artinian Gorenstein quotients $A=R/I$ having $I_2=V$, and all such Artinian quotients; and by $\Gor_{\{V\}}(T)$ and $G_{\{V\}}(T)$, respectively, those quotients for which  $I_2\cong V$ - is $\Pgl(3)$ congruent to $V$. We will also denote the by $\G_{\#6b}(T)=G_{\{V\}}(T)$ - for $ V$ of \#6b.\par
In Section \ref{4.1sec} we present some ingredients needed for the classification.  In Section \ref{4.2sec} we show the classification Theorem 5, including the Betti tables for each class. The statement of cases and as well the  proof of Theorem 5 extends from page \pageref{thm5restate} to \pageref{endproofthm5}.
\subsection{Parametrizing Artinian algebras in $G_{\{V\}}(T)$.}\label{4.1sec}
When $T=(1,3^k,1)$ with $k\ge 3$ and $A=R/I\in G_T$, then the vector space $V=I_2$ must have scheme length three, or be in case \#2a, as otherwise $H(R/(I_2))=(1,3,3,\ell,\ldots)$ where $\ell=\ell(V)$, and $H(R/I)$ would be smaller than $T$.

We will first show 
\begin{theorem}\label{dimthm} We let $d= d(V)$ be the dimension of the family of nets of conics that are $\Pgl(3)$ equivalent to $V$.  Let $T=(1,3^k,1)$ with $k\ge 2$. Then $G_V(T)$ and $G_{\{V\}}(T)$ are irreducible.\begin{enumerate}[(i).]
\item Let $T=(1,3,3,1).$  Then the family of Artinian algebras $G_{\{V\}}(T)$ has dimension $d$ if $\ell=0$, dimension $d+\ell-1$ if $1\le \ell\le 3$; and has dimension $5$ if $\ell=\infty$ (the case $V=\#$2a).
\item Let $T=(1,3^k,1)$ with $k\ge 3$. Then  $G_{\{V\}}(T)$ has dimension $d+2$ if $\ell=3$ and dimension $7$ if 
$\ell=\infty$ (case V=\#2a), and is empty otherwise.
\end{enumerate}
\end{theorem}
\noindent
{\it Proof of (i)}. For $\ell(V)=0$ we have $H(R/(V))=(1,3,3,1)$ and $G_V(T)=\Gor_V(T)$, which is irreducible of dimension $d$. For $V$ not in the class \#2a, then $1\le \ell(v)\le 3$, and the Hilbert function $H(R/(V))=(1,3,3,\overline{\ell}), \ell=\ell(V)$.  The family $G_{\{V\}}(T)$ is fibered over $\Grass_{\{V\}}$ parametrizing the nets $\Pgl(3)$-equivalent to $V$, by a fiber that is the projectivization of $R_3/\langle(V)\cap R_3\rangle$, a vector space of $\sf k$-dimension $\ell$: hence the fiber has  dimension $\ell-1$: and  $G_{\{V\}}(T)$ is irreducible of dimension $d+\ell-1$. For $V$ in the case \#2a, the Hilbert function $H(R/(V))=(1,3,3,4,5,\ldots)$ and $I_3$ is determined by a 3-dimensional quotient of $R_3/R_1V$, hence by an element of a projective 3-space $\mathbb P(R_3/R_1V)\cong \mathbb P^3$; so $G_{\#2a}(T)$ is irreducible of dimension $5$. \par\noindent
{\it Proof of (ii).} When $k\ge 3$, and $\ell=3$, then Theorem \ref{schl3thm} Equation~\eqref{schl3eqn} shows that $I_u=(V)_u$ for $2\le u\le j-1$.  Then $I_j$ is determined by adding to $(V)_j$ a two-dimensional subspace linearly disjoint from $(V)_j$: equivalent to choosing a $2$-dimensional subspace of $R_j/(V)_j$, which has dimension three. Thus, $G_V(T)$ is irreducible of dimension two, and $G_{\{V\}}(T)$ is irreducible of dimension $d+2$, as claimed.
\par
When $V=\langle x^2,xy,xz\rangle$ is \#2a in Figure \ref{1fig} the Hilbert function of $H(R/(V))=(1,3,3,4,5,\ldots)$, and $\dim  \Grass_{\{V\}}=2$. Given $V$ we first choose $f\in {\sf k}[y,z]_3$, and $H(R/(V,f))=(1,3,3,\overline{3})$ (since $f$ can have only trivial relations with the elements of $V$): here $f$ is parametrized by $\mathbb P^3$. Then we choose a two dimensional subspace $W$ of the 3-dimensional space $R_j/\langle (V,f)\cap R_j\rangle$, and take $I=(V,f,\mathcal W)$ where $\mathcal W$ in $R_j$ represents $W\mod (V,f)_j$.  Here $W$ is an element of $\mathbb P^2$. This shows  that for $V$ congruent to \#2a, $G_V(T)$ is irreducible of dimension five; then its image $G_{\#2a}(T)=G_{\{V\}}(T)$  under $\Pgl(3)$ action is irreducible of dimension $2+5=7$ when $k\ge 3$.\qquad\qquad\qquad\qquad\qquad\qquad\qquad\qquad\qquad\qquad\quad$\square$
\vskip 0.2cm\par

We may also use that when $T=(1,3,3,1)$ the fiber of  $G_{\{V\}}(T)$ over $\Grass_{\{V\}}$ is isomorphic to an open in the projective space $\mathbb P(S_3/((V)_3^\perp)$ over the dual of $R_3/R_1V$.\par
Recall that specialization of nets $V^\prime$ to $V$ must satisfy $\ell(V^\prime)\le \ell(V)$.  We see from Figure \ref{1fig} that the net $V=(xy,xz,x^2)$ (\#2a) is in the closure of all families of nets in Figure~\ref{1fig} except \#2b.
\begin{corollary}[Specialization]\label{closurecor} 
(i). Let $T=(1,3,3,1)$. Suppose the net $V^\prime$ has dimension $d(V^\prime)=d(V)+1$ and scheme length $\ell(V^\prime)=\ell(V)-1$, and assume $V$ is not in \#7b. Then the Zariski closure of $G _{\{V^\prime\}}(T)$ cannot include $G_{\{V\}}(T)$.\par
(ii). Let  $T=(1,3^k,1), k\ge 3$. Let $V=(xy,xz,x^2)$ (Case \#2a). Then $G_{\{V\}}(T),$ $ T=(1,3^k,1), k\ge 3$ is not in the closure of any $G_{\{V^\prime\}}(T)$ for which $d(V^\prime)\le 5$.\par
\end{corollary}
\begin{proof}  Corollary \ref{closurecor} (i) follows from Theorem \ref{dimthm}(i) and the scheme lengths in Figure \ref{1fig}. This is so since the dimension of both $G _{\{V^\prime\}}(T)$ and
 $G_{\{V\}}(T)$ are the same under the hypotheses, and each is irreducible.  Corollary \ref{closurecor} (ii.) follows from Theorem~\ref{dimthm}(ii). 
\end{proof}
Recall that Proposition~\ref{2aprop} shows, pertaining to (ii) above, that if the net $V$ is in \#2a,  then $G_{\{V\}}(T)$ is in the closure of $G_{\#6a}(T)$.  
\begin{remark}\label{closurerem} Part (i) of the Corollary applies, for example, to the pairs  (\#5b, \#4), (\#6b,\#5a),(\#6c, \#5a), (\#6d,\#5b), and (\#7b,c, \#6b,c) - simply by reading from Figure \ref{1fig}. Since $G_{\{V^\prime\}}(T)$ is $\Pgl(3)$ invariant, so are their closures. Since all classes but \#8b of nets of conics are  ``discrete'' (do not involve parameters), the closure $\overline{\Grass_{\{V^\prime\}}}\subset \Grass(R_2,3)$ of an isomorphism stratum $\Grass_{\{V^\prime\}}$ of nets over an algebraically closed field of characteristic zero, or of characteristic $p\ne 2,3$ is the union of certain lower strata, specified in \cite{AEI} and in our Figure~\ref{1fig}.\par

The family $G_{\{V\}}(T)$ for a fixed isomorphism class $\{V\}$ of nets of conics is irreducible by Theorem~\ref{dimthm}, but as we will see in Theorem~\ref{5thm} below, the family may contain several - or even an infinite family- of isomorphism classes of Artinian algebras.
\end{remark}
\begin{question} The above results then lead to a typical question in parametrization, analogue of one that baffled the Nice school of Hilbert scheme specialists concerning closures of families $Z_T$ of non-graded Artinian quotients of ${\sf k}[y,z]$ having a given Hilbert function $T$.\footnote{The fundamental result of the irreducibility of $\Hilb^n({\sf k}[y,z])$: that the closure of the stratum $T=(1^n)$ is all length-$n$ Artinian algebra quotients of ${\sf k}[y,z]$ had been shown by J.~Brian\c{c}on in \cite{Bri}, but the closures of other strata are even now only partially understood, see the last author's \cite{Y1,Y2,Y3}.}  Our problem here is to determine properties (as dimension) of the intersections
$\overline{G_{\{V^\prime\}}(T)}\cap G_{\{V\}}(T),$ when $V^\prime$ specializes to $V$, and the scheme length of $V^\prime$ is less than that of $V$. Examples occur in Remark \ref{closurerem} when for dimension reasons there cannot be an inclusion $\overline{G_{\{V^\prime\}}(T)}\supset G_{\{V\}}(T)$ (for example the pair (\#5b,\#4)). Are these intersections irreducible?
\end{question}\par
\subsubsection{Betti tables and cancellation.}\label{Betti1sec}
We recall that any finitely generated graded Artinian algebra $A$ of codimension three, quotient of $R={\sf k}[x_1,x_2,x_3]$, has a minimal graded finite free resolution
\begin{equation}\label{bettieq}
\mathcal F:\quad 0\to  F_3\to F_2\to F_1\to R\to A\to 0,
\end{equation}
where $F_i=\oplus_j R(-j)^{\beta_{i,j}}$, and $\beta_{i,j}=\mathrm{Tor}_i^R(A,{\sf k})_j$. We represent this in a
\emph{Betti table}, whose $(i,j)$ entry in column $i$, row $j$ is $\beta_{i,i+j}$. For example, the minimal resolution
for $A=R/I, I=(xy,xz,yz,x^3-y^3,z^3)$, the ideal in \#6a.ii. below, whose terms $F_3,F_2,F_1, F_0$ are
\footnotesize
\begin{equation}\label{bettiexeq}
 R(-6)\oplus R(-5)\to R(-5)\oplus R(-4)^3\oplus R(-3)^2\to R(-3)^2\oplus R(-2)^3\to R,
\end{equation}
\normalsize
which is encoded in the Betti diagram C, Figure \ref{BettiCfig}.\par
We will in Section \ref{bettisec} discuss problems involving limits of Betti tables. A \emph{consecutive cancellation}
from a resolution $\mathcal F$ to $\mathcal F^\prime$ is given by setting for some pair $(i,j), (i+1,j)$ of adjacent indices where $\beta_{i,j}, \beta_{i,j+1}\not=0$, the Betti indices $\beta^\prime_{i,j}=\beta_{i,j}-1$, 
$\beta^\prime_{i+1,j}=\beta_{i+1,j}-1$. In a Betti table, this will lead to two indices in 
adjacent columns and rows, one diagonally above the other, being each decreased by one. For example
the Betti diagram C of Figure \ref{BettiCfig} arises by consecutive cancellation of $\beta_{3,5}=\beta_{2,5}=1$ from Equation \eqref{bettiexeq} above and Betti diagram D of Figure \ref{BettiDfig}.  Recall that the \emph{lexicographic ideal J} of Hilbert function $H(R/J)=H$ is the unique monomial ideal $J$ first in the lex-order.  I. Peeva has shown \cite[Theorem 1.1]{Pe}
\begin{lemma}\label{Peevalem} (I. Peeva) Let $I $ be a graded ideal and $L$ be the lexicographic ideal with the same Hilbert function. The graded Betti numbers $\beta_{i,j}(S/I)$ can be obtained from the graded Betti numbers $\beta_{i,j}(S/L)$ by a sequence of consecutive cancellations.
\end{lemma}
This was extended to local algebras $A$, and a comparison between the graded Betti numbers for the associated graded algebra and $A$ by M.E. Rossi and L. Sharifan \cite{RoSh}, using negative and consecutive cancellations. 
Maria Evelina Rossi pointed out to us that $L$ above in Lemma \ref{Peevalem} may be replaced by $Lt(I)$ the leading terms of $I$ in any term order (see \cite[Proof (2) of Thm. 1.1]{Pe}).
 \vskip 0.2cm\par\noindent
 \subsection{Classification of isomorphism classes of Artinian algebras of Hilbert function $T$ by the net of conics $V=I_2$: the family $G_V(T)$.}\label{4.2sec}
 We here show \label{thm5restate}
 \par\noindent
 {\bf Theorem 5}. {\it Let $T=(1,3^k,1)$. We determine the isomorphism classes of algebras in $G_V(T)$, the family of Artinian quotients $A=R/I$ in $G_T$ for which $I_2=V$. \par
There are among the nets $V$ of scheme length 3, exactly four isomorphism classes of Gorenstein algebras in $G_V(T)$, namely
 \#6a.i,\#5a.i,\#5a.ii, and \#4.iii., (see the Betti C table in Figure~\ref{BettiCfig}). There are no further non-CI Gorenstein classes of algebras in $G_T$.}\par
 This Theorem summarizes the results of Theorem \ref{sch2algthm} (Section \ref{4.6sec}), Theorem \ref{GTclasseslength3}(Section \ref{2asec} and Theorem \ref{2aAthm}(Section \ref{4.11sec}) for $V$ of scheme length two, three and $\infty$, respectively, where we give the detailed classification, and also specify the Betti tables. A detailed statement of Theorem 5 (minus the Betti tables) is in the Introduction. The notation \#6a.i refers to a specific isomorphism class within $G_{6a}(T)$, that is given in the statement of Theorem \ref{GTclasseslength3}, Equation \eqref{isomclass6aeq}.
We assume throughout this section that $\sf k$ is closed under taking $j$-th roots, or is algebraically closed,
of characteristic not 2 or 3.

 \par\vskip 0.2cm \noindent
 \subsubsection{Classification of algebras in $\Gor(T), T=(1,3^k,1)$.}
{\it Previous classification of isomorphism types of Artinian Gorenstein (AG) algebras having Hilbert function $T=(1,3^k,1), k\ge 2$.} \par
Recall from Proposition \ref{Gotzmannprop}(ii). and Lemma \ref{regslem} that when $k\ge 4$, there is necessarily a length 3 punctual subscheme $\Z\subset\mathbb P^2$, determined as $\Z=\Spech R/(\overline{I_3})$, where $\overline{I_3}$ is the ancestor ideal of $I_3$. This occurs since $(3_3,3_4)$ (meaning a Hilbert function sequence $3$ in degree three, then $3$ in degree $4$) is a maximum growth sequence in the sense of Lemma \ref{Maclem} (minimum possible size of $\dim_{\sf k}R_1V$ given $\dim_{\sf k}V=3$).  Thus, determining $G_T$ is closely related to the study of length 3 punctual subschemes of $\mathbb P^2$. The article of A. Bernardi et al \cite[\S 4.2, Theorem 37]{BGI} classifies the isomorphism types of length 3 subschemes of $\mathbb P^n$ in characteristic zero. When  $n=2$ we may derive this classification from the scheme-length 3 cases of Figure~\ref{1fig}, together with a study of length-3 schemes that are on a line: the latter is related to  Section \ref{2asec} below. P. LeBarz and G.~Elencwajg classify triangles in $\mathbb P^2$:  they determine the cohomology ring structure of $\Hilb^3 (\mathbb P^2)$ \cite{EL} - see aso E. Vasserot's more general \cite{Va}.\par
 G. Casnati, J. Jelisiejew, and R. Notari determine, for arbitrary characteristic, the three isomorphism classes of forms $F\in {\sf k}[X,Y,Z]_j$ whose apolar algebra $A=R/\mathrm{Ann} F$ has Hilbert function $H(A)=(1,3^{k},1) $ when $k\ge 3$ \cite[Proposition 4.9]{CJN}. 
These are
\begin{equation}X^j+Y^j+Z^j, X^{j-1}Y+Z^j, X^{j-2}(XZ+Y^2).
\end{equation}
The statements there are given for characteristic zero, but the arguments hold for $\sf k$ algebraically closed of any characteristic $p\not=2,3$.
J. Jelisiejew \cite{Je} determines the isomorphism types of AG algebras $A$ of Hilbert function $H(A)=(1,3,3,3,1)$ including those for non-graded $A$.  \par
When $T=(1,3^k,1)$ with $k\ge 3$ then all Gorenstein algebras $A=R/I$ of Hilbert function $T$ have five generators. The generators are comprised of the net, which must have scheme length three, and two in the socle degree $j=k+1$. When $k=2$ the ideal $(V)$ of a net may define either a complete intersection, or we may consider Gorenstein algebras  $A=R/I$ whose ideal has five generators - two of degree four (the maximum number possible for $I$ having initial degree two. We see this in our classification Theorem 5. \par\vskip 0.2cm\noindent
\vskip 0.2cm\noindent

We next classify all graded algebras $A\in G_T$, for these Hilbert functions, and specify the dimension of their isomorphism classes, as well as their Betti tables, dividing our classification according to the scheme length of $V=I_2$.

\subsubsection{Isomorphism classes in $G_{\{V\}}T$ when $V=I_2$ defines a length $0$ or $1$ scheme.}
Evidently, these ideals are just $I=(V,\m^4)$ so their isomorphism classes and dimensions are those of the corresponding nets $V=I_2$.  When the scheme length of $V$ is $0$ then $I_2$ defines a complete intersection,
with CI betti table of Figure \ref{BettiCIfig}: these are \#8b,\#8c,\#7c,\#6d.
\par

The Betti resolutions in both \#8a, \#7b, are the table SL1 of Figure \ref{BettiSL1fig}.
\subsubsection{Isomorphism classes in $G_{\{V\}}T$ when $V=I_2$ defines a length two scheme.}\label{4.6sec}
 Here $T=(1,3,3,1)$ is the only Hilbert function possible. We give isomorphism classes $G_{\{V\}}T$ of Artinian algebras where the net $V=I_2$ defines a length two scheme by cases, according to isomorphism class of the net $V$ in  Figure \ref{1fig} (derived from \cite[Tables 1 and 5]{AEI}). These are the cases \#7a,\#6b,\#6c, and \#5b.\par
 \begin{theorem}\label{sch2algthm} The isomorphism classes of graded Artinian algebras of Hilbert function $T=(1,3,3,1)$ for which $V=I_2$ has scheme length two, and their Betti tables are listed as follows:\begin{enumerate}[(i).]
 \item Case $V$ in class \#7a: two classes, Equation \eqref{7aeq};
 \item Case $V$ in class \#6b. There is a one-parameter family of non-isomorphic Artinian algebras, parametrized by $\mathbb P^1$, having  $V=I_2$ in \#6b (Conclusion \ref{6bconclude});
 \item Case $V\in $\#6c: three classes, Equation \eqref{6ceqn};
 \item Case $V\in$ \#5b: three classes, Equation \eqref{5beqn}.
 \end{enumerate}
 There are no Gorenstein algebras having such $V=I_2$, that is $\Gor_V(T)$ is empty for each $V$ of scheme length two.
 \end{theorem}
\begin{proof}  We work out the proof by cases.\par\noindent
{\bf Case \#7a}: Here  $V=\langle x^2+yz,xy,xz\rangle $; any algebra in $G_{\{V\}}(T)$ has the form $A(\alpha,\beta)=R/I(\alpha,\beta), I(\alpha,\beta)=(V,f(\alpha,\beta),\m^4)$ where $ f= \alpha y^3+\beta z^3$.  $R_1V^\perp=\langle Z^3,Y^3\rangle$. Here $V$ is invariant under scaling $\omega_ax\to x',y\to a y' ,z\to a^{-1}z'$, and also under $\theta_b: x\to bx' ,y\to b^2y',z\to z' $,
 and under $\sigma\theta: x\to b x' ,y\to y' , z\to b^2 z' .$\par
 When both $\alpha,\beta\neq 0$, we may scale $z$ by a cube root of $\beta/\alpha$, to obtain $A(\alpha,\beta)\cong A(1,1)$ giving \# 7a.i, of Betti table A from Figure \ref{BettiAfig}.  When one of $\alpha,\beta=0$, then, by symmetry $y\to z, z\to y$, there is a single isomorphism class \#7a.ii. of Betti table B from Figure \ref{BettiBfig}.
 Also, there are no Gorenstein elements of $G_V(T)$.
 \begin{align} 
   \#7a.i: I&=(z^2+xy,yz,xz,y^3+z^3,\m^4), \qquad \text { dim  }8;\notag\\
 \#7a.ii:  I&=(z^2+xy,yz,xz,x^3,\m^4). \qquad \text { dim }7.\label{7aeq}
  \end{align}
  \vskip 0.2cm\par\noindent
 {\bf Case \#6b}: Here any algebra in $G_{\{V\}}(T)$ has the form $A=R/I, I=(V,f,\m^4)$ where $V=\langle x^2,(y+z)z,xy\rangle$ and $ f\in R_3/R_1V\cong \langle y^3,z^3\rangle$, note
 $R_1V^\perp=\langle Y^3, Y^2Z-YZ^2+Z^3\rangle$. Notice that the stabilizer  $G=
\Stab_{\Pgl(3)} V$ in $\Pgl(3)$ is two
 dimensional, as the dimension of the orbit of \#6b is six, and the dimension $\Pgl(3)$ is 8. The stabilizer $G$ is generated by the two
 scalings:  $x\to ax$ and $y,z \to by,bz, a,b\neq 0$. These each act on $f(\alpha,\beta)=\alpha z^3+\beta y^3$ by scaling it, by $1$, or $b^3$, respectively. Thus, when $\alpha\beta\neq 0$ we may take  $\alpha=1, f(1,\beta)=z^3+\beta y^3$, and we have a one-dimensional parametrized family of non-isomorphic algebras $$A_V(1,\beta)=R/(V,z^3+\beta y^3,\m^4) \text{ with }I_2=V.$$ 
 For $\beta\neq 0$, $A_V(1,\beta)$ has the Betti diagram
   A of Figure~\ref{BettiAfig}.  This is since the socle has an element of degree 3 and another of degree 2 (the last column of diagram A); the three generators in degree two have a single relation, requiring the new generator $z^3+\beta y^3$ of degree 3, in order to match the Hilbert function: this confirms Betti diagram  A of Figure~\ref{BettiAfig}.

The Betti diagram A of Figure \ref{BettiAfig} occurs also for $A_V(0,1)$.\vskip 0.2cm\par
For $\beta=0$ the algebra $ A_V(1,0)=R/(V,z^3,\m^4)$ has the Betti diagram B of Figure \ref{BettiBfig}.  Note, there are no  Artinian Gorenstein algebras in $\Gor_V(T)$ - this can be seen from the Betti diagrams, or from $F\in (V)_3^\perp\in \langle Y^3,Z^3\rangle$, so
$R_1\circ F\subset \langle Y^2,Z^2\rangle$ has dimension two not three.
\begin{conclusion}\label{6bconclude}[Projective line of isomorphism classes] For $V=I_2$ isomorphic to \#6b, there is a projective line $\beta\in\mathbb P^1$ of non-isomorphic algebras $A(1,\beta)$, elements of $G_V(T), T=(1,3,3,1)$, including $A_V(1,0)$ and $A_V(0,1)\equiv A_V(1,\infty)$.  The Betti diagram A of Figure~\ref{BettiAfig} occurs for all classes except $A_V(1,0)$, which has Betti diagram B of Figure \ref{BettiBfig}. There are no Gorenstein elements of $G_V(T)$.  
 \end{conclusion}
\subsubsection{Betti diagrams for $A\in G_T, T=(1,3,3,1)$.}

\begin{figure}
 \begin{small}
$$
\vbox{\offinterlineskip 
\halign{\strut\hfil# \ \vrule\quad&# \ &# \ &# \ &# \ &# \ &# \
&# \ &# \ &# \ &# \ &# \ &# \ &# \
\cr
total&1&3&3&1\cr
\noalign {\hrule}
0&1 &--&--&--&\cr
1&--&3 &0& -- &\cr
2&--& &3&0&\cr
3&--&&& 1 &\cr
\noalign{\smallskip}
}}
$$
\end{small}
\vskip -0.8cm
\caption{Betti diagram CI, \#8b,\#8c,\#7c,\#6d}\label{BettiCIfig}
\end{figure}
\begin{figure}
 \begin{small}
$$
\vbox{\offinterlineskip 
\halign{\strut\hfil# \ \vrule\quad&# \ &# \ &# \ &# \ &# \ &# \
&# \ &# \ &# \ &# \ &# \ &# \ &# \
\cr

total&1&4&6&3\cr
\noalign {\hrule}
0&1 &--&--&--&\cr
1&--&3 && -- &\cr
2&--& &4&2&\cr
3&--&1&2& 1 &\cr
\noalign{\smallskip}
}}
$$
\end{small}
\vskip -0.8cm
\caption{Betti diagram SL1 (Scheme length 1) \#8a,\#7b.}\label{BettiSL1fig}
\end{figure}
 \begin{figure}
 \begin{small}
$$
\vbox{\offinterlineskip 
\halign{\strut\hfil# \ \vrule\quad&# \ &# \ &# \ &# \ &# \ &# \
&# \ &# \ &# \ &# \ &# \ &# \ &# \
\cr

total&1&4&5&2\cr
\noalign {\hrule}
0&1 &--&--&--&\cr
1&--&3 &1& -- &\cr
2&--&1 &3&1&\cr
3&--&-- &1& 1 &\cr
\noalign{\smallskip}
}}
$$
\end{small}
\vskip -0.8cm
\caption{Betti diagram A, \#7a.i.,\#6b$(1,\beta)$, \#5b i,ii}\label{BettiAfig}
\end{figure}

\begin{figure}
 \begin{small}
$$
\vbox{\offinterlineskip 
\halign{\strut\hfil# \ \vrule\quad&# \ &# \ &# \ &# \ &# \ &# \
&# \ &# \ &# \ &# \ &# \ &# \ &# \
\cr
total&1&5&7&3\cr
\noalign {\hrule}
0&1 &--&--&--&\cr
1&--&3 &1& -- &\cr
2&--&1 &4&2&\cr
3&--&1 &2& 1 &\cr
\noalign{\smallskip}
}}
$$
\end{small}
\vskip -0.8cm
\caption{Betti diagram B,\#7a.ii, \#6b $A_V(0,1)$,\#5b.iii.}\label{BettiBfig}
\end{figure}
\begin{figure}
 \begin{small}
$$
\vbox{\offinterlineskip 
\halign{\strut\hfil# \ \vrule\quad&# \ &# \ &# \ &# \ &# \ &# \
&# \ &# \ &# \ &# \ &# \ &# \ &# \
\cr
total&1&5&5&1\cr
\noalign {\hrule}
0&1 &--&--&--&\cr
1&--&3 &2& -- &\cr
2&--&2 &3&&\cr
3&--&&& 1 &\cr
\noalign{\smallskip}
}}
$$
\end{small}
\vskip -0.8cm
\caption{Gorenstein Betti diagram C \#6a.i,\#5a.i,ii.}\label{BettiCfig}
 For Betti C$_j$ (socle degree $j$) move rows $(2,3)$ to rows $\{j-1, j\}$.
\end{figure}

\begin{figure}
 \begin{small}
$$
\vbox{\offinterlineskip 
\halign{\strut\hfil# \ \vrule\quad&# \ &# \ &# \ &# \ &# \ &# \
&# \ &# \ &# \ &# \ &# \ &# \ &# \
\cr
total&1&5&6&2\cr
\noalign {\hrule}
0&1 &--&--&--&\cr
1&--&3 &2& -- &\cr
2&-&2 &3&1&\cr
3&-&&1& 1 &\cr
\noalign{\smallskip}
}}
$$
\end{small}
\vskip -0.8cm
\caption{Betti diagram D,\#6a.ii;\#5a.iii,iv,vii;\#4.ii;\#2bi,ii.}\label{BettiDfig}
 For Betti D$_j$ (socle degree $j$) move rows $(2,3)$ to $\{j-1, j\}$.
\end{figure}

\begin{figure}
 \begin{small}
$$
\vbox{\offinterlineskip 
\halign{\strut\hfil# \ \vrule\quad&# \ &# \ &# \ &# \ &# \ &# \
&# \ &# \ &# \ &# \ &# \ &# \ &# \
\cr
total&1&6&8&3\cr
\noalign {\hrule}
0&1 &--&--&--&\cr
1&--&3 &2& -- &\cr
2&--&2 &4&2&\cr
3&--&1&2& 1 &\cr
\noalign{\smallskip}
}}
$$
\end{small}
\vskip -0.8cm
\caption{Betti diagram E, \#6a.iii.,\#5a.v,vi,4.iii.,\#2b.iii.}\label{BettiEfig}
 Betti E$_j$ (socle degree $j$) move rows $(2,3)$ to rows $\{j-1, j\}$.
\end{figure}

\begin{figure}
 \begin{small}
$$
\vbox{\offinterlineskip 
\halign{\strut\hfil# \ \vrule\quad&# \ &# \ &# \ &# \ &# \ &# \
&# \ &# \ &# \ &# \ &# \ &# \ &# \
\cr
total&1&6&8&3\cr
\noalign {\hrule}
0&1 &--&--&--&\cr
1&--&3 &3&1 &\cr
2&-&3 &4&1&\cr
3&-&&1& 1 &\cr
\noalign{\smallskip}
}}
$$
\end{small}
\vskip -0.8cm
\caption{Betti diagram G \#2a.i,ii Equation \eqref{2a1eqn}}\label{BettiGIfig}
\end{figure}

\begin{figure}
 \begin{small}
$$
\vbox{\offinterlineskip 
\halign{\strut\hfil# \ \vrule\quad&# \ &# \ &# \ &# \ &# \ &# \
&# \ &# \ &# \ &# \ &# \ &# \ &# \
\cr
total&1&7&10&4\cr
\noalign {\hrule}
0&1 &--&--&--&\cr
1&--&3 &3&1 &\cr
2&-&3 &5&2&\cr
3&-&1&2& 1 &\cr
\noalign{\smallskip}
}}
$$
\end{small}
\vskip -0.8cm
\caption{Betti diagram H \#2a.iii (Lex ideal) Equation \eqref{2a1eqn}.}\label{BettiHIfig}
\end{figure}
\newpage
\noindent
 \vskip 0.2cm\par\noindent
{\bf Case \#6c}: We have $V=\langle xz,yz, x^2+z^2\rangle$ and $R_1V^\perp=
\langle XY^2,Y^3\rangle$. Any element $A=R/I$ of $G_V(T)$ satisfies $I=(V,f(\alpha,\beta),\m^4)$ where $f(\alpha,\beta)=\alpha y^3-\beta xy^2$. \par\noindent
Case (i). When both $\alpha,\beta\neq 0$, we have 
\begin{align}I(\alpha,\beta)&=(xz,yz,x^2+z^2,\alpha y^3-\beta xy^2,\m^4)=I(1,\lambda)\notag\\
&=(xz,yz,x^2+z^2,f(1,\lambda)=y^3-\lambda xy^2,\m^4),
\end{align} where  $\lambda=\beta/\alpha.$   By scaling $(x,z)\to \lambda(x',z') $ and $y\to y' $ we have  $I(1,\lambda)=I(1,1)$, and its Betti table is that of Figure \ref{BettiAfig}. 
We have three isomorphism classes:
\begin{align}
I(1,1)&=(xz,yz,x^2+z^2, y^3-xy^2,\m^4),\, \text { dimension } 7, \text { Betti table A }\notag\\
I(1,0)&=(xz,yz,x^2+z^2, y^3,\m^4),\, \text { dimension } 6,\text { Betti table A }\notag\\
I(0,1)&=(xz,yz,x^2+z^2, x^2y,\m^4),\, \text { dimension } 6. \text { Betti table B }.\label{6ceqn}
\end{align}
 There are no Gorenstein elements of $G_{\#6c}(T)$.\vskip 0.2cm\par\noindent
{\bf Case \#5b}:  We have $V=\langle xy,y^2,z^2\rangle $ and $R_1V^\perp=\langle X^3,X^2Z\rangle$, so $A=R/I$ with $I(\alpha,\beta)=(V,f(\alpha,\beta),\m^4)$ with $f(\alpha,\beta)=\alpha x^3+\beta x^2z$. Note that there can be no Gorensteins: $\Gor_V(T)$ is empty, as $R_1\circ (R_1V)^\perp\subset \langle X^2,XZ\rangle.$.\par  When $\alpha \neq 0$ we have $I(\alpha,\beta)=I(1,\lambda), \lambda=\beta/\alpha \in {\sf k}$.  Here $f=x^3+\lambda x^2z$ for $\lambda\ne 0$ can be scaled to $\lambda=1$ by scaling $x$. Thus, we have three isomorphism classes.  Note that for 5b.ii the ideal $I_{\le 3}$ contains  powers of all the variables, so defines an Artinian quotient of $R$, in contrast to 5b.iii: thus these cases  are non-isomorphic. By dimension they cannot be isomorphic to \#5b.i .\end{proof}
Thus we have - here $V =\langle xy,y^2,z^2\rangle$ -
\begin{align}
\#5b.i:\quad & I(1,1)=(V,x^3+x^2z,\m^4), \quad \text{ dim } =6, \text{ Betti table A;} \notag\\
\#5b.ii:\quad &I(1,0)=(V,x^3,\m^4),  \quad \text{ dim } =5,\text{ Betti table A;}\notag\\
\#5b.iii:\quad &I(0,1)= (V, x^2z,\m^4),  \quad \text{ dim } =5, \text{ Betti table B.}\label{5beqn}
\end{align}
\vskip 0.2cm\par\noindent
{\bf Question.} We have by Theorem \ref{dimthm} that the dimension of \#5b.i is  6; and we can check that \#5b.ii and \#5biii both have dimension 5. Thus, $A\in G_{\{ \#6d\}}(T)$ of scheme length zero and dimension 6 must specialize to either or both of \#5b.ii or \#5b.iii.  Which?\par
\vskip 0.2cm
\subsection{Isomorphism classes in $G_T$ when $I_2$ defines a length 3 scheme, $T=(1,3^k,1)$.}\label{2asec}
From Figure \ref{1fig} these are the cases $V=I_2$ from \#6a,\#5a,\#4, and \#2b.  Recall from Theorem \ref{schl3thm}, Equation \eqref{schl3eqn} that $I_u=(V)_u$ for $2\le u\le j-1$. Thus, the fiber of $G_{\{V\}}(T)$ over a fixed $V=I_2$ parametrizes algebras $A$ entirely determined by the pair $(V,\langle F\rangle)$, where $\langle F\rangle\in (V_j)^\perp\subset S_j$,  as $I_j=(\Ann F)_j$. The three-dimensional vector spaces  $(V_j)^\perp$ are given in Figure \ref{netssl3}. We now just need to
\begin{enumerate}[(i).]
\item parametrize the choices of $\langle F\rangle$ from these.
\item Determine their isomorphism classes.
\end{enumerate}
\begin{theorem}\label{GTclasseslength3} The isomorphism classes of Artinian algebras of Hilbert function
$H=(1,3^k,1), k\ge 2)$, for which $V=I_2$ has scheme length three are given as follows:\begin{enumerate}[(i).]
\item Case $V=I_2=(xy,xz,yz)$ (\#6a), 3 classes, Equation \eqref{isomclass6aeq}.
\item Case $V=I_2=(xy,xz,z^2)$ (\#5a), 7 classes, Equation \eqref{5aisomeq}.
\item Case $V=I_2=(x^2+yz,xy,y^2)$ (\#4), 3 classes, Equation~\eqref{4isomeq}.
\item Case $V=I_2=(x^2,xy,y^2)$ (\#2b), 3 classes, Equation~\eqref{2bisomeq}.
\end{enumerate}
\end{theorem}
{\it Proof and discussion of cases}: \par\noindent
{\bf
\#6a Here $V=(xy,xz,yz)$} defining the scheme $\Z=(1,0,0)\cup (0,1,0)\cup (0,0,1)$, and $R_{j-1}V^\perp=\langle X^j,Y^j, Z^j\rangle$. Assuming that ${\sf k}$ is closed under cube roots, we have up to isomorphism, three classes determined by choosing $I_j^\perp=X^j+\alpha Y^j+\beta Z^7$ (case 6a.i where $A$ is Gorenstein), or  ${I_j}^\perp=X^j+\alpha Y^j$ (case 6a.ii.) or $I_3^\perp = X^j$ (case 6a.iii.). These lead to the families,
\begin{align}
\#6a.i: &(V, \alpha x^j-y^j, \beta x^j-z^j).\text { Gorenstein family, dim } 8,\notag\\
\#6a.ii: &(V,\alpha x^j-y^j,z^j). \quad\text { dim } 7, \notag\\
\#6a.iii: & (V,y^j,z^j,x^{j+1}).\quad \text { dim } 6.\label{6aclass}
\end{align}
whose isomorphism classes, assuming ${\sf k}$ is closed under $j$-th roots, are
\begin{align} 6a.i: & (V,x^j-y^j,x^j-z^j); \,6a.ii: (V, x^j-y^j,z^j)\text { and }\notag\\
 6a.iii:\,& (V,y^j, z^j,x^{j+1}).\label{isomclass6aeq}
\end{align}

When $k=2$ \#6a.i has the Betti diagram C of Figure \ref{BettiCfig}, 6a ii has the Betti diagram D of Figure \ref{BettiDfig}, while \#6a.iii. has the Betti diagram E of Figure \ref{BettiEfig}.\par\vskip 0.2cm\noindent
{\bf \#5a:  $V=(xy,xz, z^2)$} so $(R_{j-1}V)^\perp=\langle X^j,Y^j,Y^{j-1}Z\rangle$. \par\noindent
So $I(\alpha,\beta,\gamma;j)=(V,W(\alpha,\beta,\gamma;j),\m^{j+1})$ for some $\alpha,\beta,\lambda \in {\sf k}$, where
\begin{align}
W(\alpha,\beta,\gamma;j)&=(\Ann F)\cap R_j,F=\alpha X^j+\beta Y^j+\lambda Y^{j-1}Z\notag\\
&=\langle \beta x^j-\alpha y^j, \gamma x^j-\alpha y^{j-1}z, \gamma y^j-\beta y^{j-1}z\rangle.\label{7aeqn}
\end{align}
If $\sf k$ is closed under $j$-th roots, we may, up to isomorphism, scale $x,y,z$ by $j$-th roots, keeping $V$ fixed.
It follows that we have the following isomorphism classes, where each named parameter among $\alpha,\beta,\gamma$ is non-zero and where $j\ge 3$.
\begin{align}
\#5a.i:&I(\alpha,\beta,\gamma;j)\equiv I(1,1,1;j), \text { for } \alpha,\beta,\gamma\neq 0 ,\notag\\
&I(j)=(V,x^j-y^j,y^{j-1}z-y^j),  \text { Betti C$_j$, dim } 7;\quad\notag\\
\#5a.ii: &I(\alpha,0,\gamma;j)\equiv I(1,0,1;j), \notag\\
&I=(V,x^j-y^{j-1}z,y^j), \text { Betti C$_j$, dim }  6; \notag\\
\#5a.iii: &I(\alpha,\beta,0;j)\equiv I(1,1,0;j),\notag\\
&I=(V,x^j-y^j, y^{j-1}z), {\text { Betti D$_j$, dim }} 6;\notag\\
\#5a.iv: &I(0,\beta,\gamma;j) \equiv I(0,1,1;j),\notag\\
&I= (V, y^j-y^{j-1}z,x^j), \text { Betti D$_j$, dim } 6; \notag\\
\#5a.v:&  I(\alpha,0,0)= I(1,0,0),\notag\\
&I=(V,y^3, y^{j-1}z,x^{j+1}), \text { Betti E$_j$, dim }5; \notag\\
\#5a.vi:&  I(0,\beta,0;j)= I(0,1,0;j),\notag\\
&I=(V, x^j,y^{j-1}z,y^{j+1}), \text { Betti E$_j$, dim} 5;\notag\\
\#5a.vii:&I(0,0,\gamma;j)=I(0,0,1;j), \notag\\
&I=(V, x^j,y^j) \text { Betti D$_j$, dim } 5.\label{5aisomeq}
\end{align}
\newpage\noindent
 {\bf \#4: $V=\langle x^2+yz,xy, y^2\rangle$,} then $R_1V^\perp=\langle Z^3,XZ^2,X^2Z-YZ^2\rangle.$ We know that $V$ defines a length three scheme in $\mathbb P^2$, supported on $(0,0,1)$; the ideal of the scheme $\Z$ in
 ${\sf k}\{x,y\}$ is $(y+x^2, yx, y^2)=(y+x^2,x^3)$, which is curvilinear. The following Lemma arose from considering the automorphisms of the scheme $\Z$: the group $ \Phi_{\underline{t}}$ describes all the automorphisms of $\mathbb P^2$ 
 that fix the curvilinear length three scheme defined by $(V)$.\par
 \begin{lemma} Let $\underline{t}=(\alpha,\beta,\gamma,\lambda)$ with $\alpha\not=0$. Then the following four-dimensional parameter space $\Phi_{\underline {t}}=\{\phi_{\underline{t}}\}$ of algebraic automorphisms of  ${\sf k}[x,y,z]$ fixes $V$: 
 we define $\phi_{\underline{t}}$ by\par
 $\begin{array}{cc}
\phi_{\underline{t}}(x)&=\quad x+\lambda y\notag\\
 \phi_{\underline{t}}(y)&=\quad(1/\alpha)y\notag\\
 \phi_{\underline{t}}(z)\quad&=\quad\gamma x+\beta y+\alpha z\\
 \end{array}\qquad =\begin{pmatrix}1&\lambda&0\\
 0&(1/\alpha)&0\\
      \gamma&\beta&\alpha\\
      \end{pmatrix}\cdot\begin{pmatrix}x\\y\\z\\\end{pmatrix}$.
 \end{lemma}
 \begin{proof}[Proof of Lemma] Just check that $\phi_{\underline{t}}(V)=V$.
 \end{proof}
 Let $I\in G_V(T)$. Then $I=(V,f_4,f_5,\m^4)$ where $\m=(x,y,z)$ and
 \begin{align}
 f_4&=a_1xz^2+a_2yz^2+a_3z^3\notag\\
f_5&=b_1xz^2+b_2yz^2+b_3z^3,\label{Weqn}
 \end{align}
 with $\dim_{\sf k}\langle f_4,f_5\rangle=2$. It is clear that a minimal set of generators for $I\in G_V(T)$ will have five or six elements.\par
 We will use the following in studying the oribit of $\Phi_{\overline{t}}$ acting on an ideal $I$.
  Note that modulo $R_1V$ we have
 \begin{align}\phi_{\underline{t}} (xz^2)&=\alpha^2xz^2+(\lambda \alpha^2-2\alpha\gamma)yz^2\notag\\
\phi_{\underline{t}}(yz^2)&=\alpha yz^2\notag\\
\phi_{\underline{t}}(z^3)&=3\alpha^2\gamma xz^2+3(\alpha^2\beta-\alpha\gamma^2)yz^2+\alpha^3z^3.
 \end{align}
 \begin{enumerate}[(1).]
 \item If $a_1=b_1=0$, then $I=(V,yz^2,z^3)$, is not Gorenstein: $A=R/I$ has a two dimensional socle $\langle xz^2,x^2-yz\rangle$, and $A$ has the Betti tabel D of Figure \ref{BettiDfig}. One can see that 
 \begin{align} \phi_{\underline{t}}(I)&=(V,yz^2,3\gamma xz^2+\alpha z^3)\notag\\
 &=(V,yz^2, xz^2+\nu z^3)
 \end{align}
 so the orbit $\Phi_{\underline{t}}(V)$ has 
 dimension one, so we have a five-dimensional isomorphism class.
 \item Suppose $a_1\neq 0$ in Equation \eqref{Weqn}. Then we may set
 $f_4=xz^2+a_2yz^2+a_3z^3; \, f_5=b_2yz^2+b_3z^3.$
 \begin{enumerate}[(a).] 
 \item if $b_3=0$ and $a_3\not=0$, then $I=(V,yz^2,xz^2+a_3z^3)$. One can see that setting
 $\alpha=a_3, 3\gamma=1$ this ideal belongs to $\Phi_{\underline{t}}(V,yz^2,z^3)$, which is (1) above.
 \item  if $b_3=0 $ and $a_3=0$, then $I=(V,xz^2,yz^2,z^4)$. Here $A=R/I$ is not Gorenstein:
 its socle $\langle z^3,x^2-yz, xz\rangle $ has dimension three, and it has Betti diagram E of Figure~\ref{BettiEfig}.
 We have $\Phi_{\underline{t}}(I)=I$, so the fiber of the orbit over $V$ has dimension zero, so the orbit of $A$ has dimension four.
 \item if $b_3\not=0$, then we may set\par
 $f_4=xz^2+a_2yz^2,\quad  f_5=b_2yz^2+z^3$. In this case the ideal $I=(V,f_4,f_5)$ is 
 Gorenstein with Betti diagram C of Figure \ref{BettiCfig}, and Macaulay dual generator $F=a_2XZ^2+X^2Z-YZ^2+b_2Z^3$.
  Such an ideal is in the orbit of $(V,xz^2,z^3)$ under $\Phi_{\underline{t}}$, the fiber of the orbit  has
  dimension two, parametrized by $a_2,b_2$; or also because this is a generic choice of $\langle f_4,f_5\rangle$ from a three-dimensional space $R_3/R_1V$.  The orbit of $A$ has dimension six.
 \end{enumerate}
 \end{enumerate}
 So we have three isomorphism classes in $G_{\{V\}}(T), V=(x^2+yz,xy,y^2)$ (take $a_2,b_2,\nu\not=0$, this is \#4 in Figure \ref{1fig}). We obtain their dimensions by adding the base dimension of four to the number of parameters in the fiber:
 \begin{align}
 4.i.\,\, I&=(V,xz^2+a_2yz^2, b_2yz^2+z^3), \text { dim. 6, Betti C (Gorenstein);}\notag\\
 4.ii.\,\,  I&=(V,yz^2, xz^2+\nu z^3), \text { dim 5, Betti D; }\notag\\
 4.iii. \, \, I&=(V,xz^2,yz^2,z^4), \text { dim 4, Betti E.}\label{4isomeq}
  \end{align}
 The analysis of $G_V(T)$ for $V=I_2=$\#4 is similar for $T=(1,3^k,1)$ and leads to analogous families, of Betti tables
  C$_j$, D$_j$, and E$_j$, respectively.
  \vskip 0.2cm\par
  \#2b. $V=(x^2,xy,y^2)=\Ann \langle  XZ,YZ,Z^2\rangle$, and  $(R_1V)^\perp=\langle XZ^2,YZ^2,Z^3\rangle.$ 
  So $V=\Ann \langle ZS_1\rangle\cap R_2$, and $ R_1V=\Ann Z^2S_1\cap R_3$, each depending on a 2-dimensional choice of $Z\in \mathbb P(S_1)$.\par\vskip 0.2cm\noindent
  {\bf Case $T=(1,3,3,1)$.} Then $I_3^\perp =\langle F\rangle$ and $I=(V,\Ann F\cap R_3,\m^4).$ We may take $F=L Z^2+\gamma Z^3$, where $L=\alpha X+\beta Y$.  But up to isomorphism $L=X$ or $L=0$ and we can scale $Z$, so up to isomorphism we have $F=XZ^2+Z^3$, or $F=XZ^2$, or $F=Z^3$,  and we have, respectively, up to isomorphism,
 \begin{align} \# 2b.i.& (V,xz^2-z^3,yz^2), \text { Betti D, dim 4; 
 }, \notag\\
 \# 2b.ii. &V, xz^2,z^3), \text{ Betti D, dim 3; } \notag\\
\# 2b.iii. & (V,xz^2,yz^2,z^4), \text { Betti E, dim 2.}\label{2bisomeq}
 \end{align}
 For the dimensions it suffices to consider group action on the dual generator $F=R_3^\perp$. \vskip 0.2cm\noindent
 {\bf Case $T=(1,3^k,1), k\ge 3$}
Here $(V^\perp)_j=\langle XZ^{j-1}, YZ^{j-1},Z^j\rangle$ and the characterization of isomorphism types is
a straightfoward generalization of the case $k=2$. Here the diagram Betti $D_j$ adjusts the degree of the later generators from three, four to $j,j+1$. We have 

 \begin{align} \# 2b.i.& (V,xz^{j-1}-z^j,yz^{j-1}), \text {Betti D$_j$, dim 4; 
 }, \notag\\
 \# 2b.ii. &V, xz^{j-1},z^j), \text{ Betti D$_j$, dim 3; } \notag\\
\# 2b.iii. & (V,xz^{j-1},yz^{j-1},z^j), \text { Betti E$_j$, dim 2.}\label{2bisombeq}
\end{align}

    \vskip 0.2cm
 \subsection{Isomorphism classes in $G_T, k\ge 2$ when $V=I_2$ defines a line (case \#2a).}\label{4.11sec}
 Here, up to isomorphism we may take $V=x R_1=\langle x^2,xy,xz\rangle$. The stabilizer of $V$ in $\Pgl(3)$ is
\begin{align} x^\prime &= x\notag\\
y^\prime &=ux+ay+bz\notag\\
z' &=vx+cy+dz, \label{stab2a}
\end{align}
where $(u, v, a, b, c, d) \in {\sf k}^6$ and $ ad-cd\not=0$.\par

  We have shown in Section \ref{2aasec}, Proposition \ref{2aprop} that each algebra $A\in G_V(T)$ is in the closure of $\Gor_{V^\prime}(T) $ where $V^\prime $ is from \#6a. We now determine their isomorphism classes.
  We let $R^\prime = {\sf k}[y,z]$ and $S^\prime={\sf k}_{DP}[Y,Z]$.\vskip 0.2cm\noindent
{\bf Consider first $T=(1,3,3,1)$.} Then $I_3=\Ann F\cap R_3$, with $F\in S^\prime_3=\langle Y^3,Y^2Z,YZ^2,Z^3\rangle$: it is well known that up to isomorphism, for $\sf k$ algebraically closed, then $F\in S^\prime_3$ satisfies $F\cong ZY(Z+Y), YZ^2$ or $Z^3$, whose isomorphism classes are of dimension three (open in $\mathbb P^3=\mathbb P(S^\prime_3)$), two (choose $Y,Z\in S^\prime_1$ up to non-constant ${\sf k}$-multiple), or one (choose $Z$ up to non-zero constant multiple), respectively.\footnote{The three cases correspond to $F$ having three distinct linear factors, in which case we may scale them so we have $F=ZY(Z+Y)$, or two distinct factors - one a  square, so isomorphic to $YZ^2$, or $F$ being a perfect cube.} We have the following result, where the fiber dimension is that of the family $G_V(T)$ over a fixed $V$. 
\par
\begin{lemma}\label{2aalem} Let $T=(1,3,3,1)$ and suppose $V=\langle x^2,xy,xz\rangle$, \#2a. Then there are three isomorphism classes of algebra $A=R/I$ in $G_V(T)$, namely
\begin{align}
2a.i.\,  I&=(x^2,xy,xz,yz^2-y^2z,y^3,z^3), \text {fiber dim 3, Betti G;}\notag\\
2a.ii.\,  I&=(x^2,xy,xz,y^2z,y^3,z^3),  \text{ fiber dim 2, Betti G;}\notag\\
2a.iii.\,   I&=(x^2,xy,xz,y^3,y^2z,yz^2,z^4), \text  { (lex ideal) fiber dim 1, Betti H.}\label{2a1eqn}
\end{align}
\end{lemma}
\begin{proof} The given ideals satisfy $I=(V, W, \m^4), $ where $W=\Ann F\cap R^\prime_3$; this also gives the dimensions of the isomorphism classes.
\end{proof}
\noindent
{\bf  Consider $T=(1,3^k,1), k\ge 3$} with $I_2=V=(x^2,xy,xz)$. Then, since $\cod R_1V=4$ in $R_3$, the ideal $I\in G_V(T)$ must have a degree-three generator, which we may consider to be  $f\in R_3/R_1V\cong \langle y^3,y^2z,yz^2,z^3\rangle= R^\prime_3$. As before for $F\in S^\prime_3$, up to isomorphism of $R^\prime$ we may assume that $f=yz(y-z), yz^2$, or $z^3$, choices of fiber dimension $3,2$ and $1$ respectively, given $V$. The Hilbert function $H(R/(V,f))=(1,3,3,\overline{3})$ as $(1,3,3,3,4,\ldots)$ is not possible by Macaulay's bound Equation \eqref{Macaulayeq} and Lemma~\ref{regslem}. We then must choose generators $g,h\in R_j/(V,f)_j$, so $I=(V,f,g,h, \m^{j+1}$; since $(V)\cap R_j\supset xR_{j-1}$, we may assume that $g,h\in R^\prime_j$ and are determined by a Macaulay dual generator $F\in S^\prime_j$ that is annihilated by $f$. We next consider the above three isomorphism classes for $f$ and specify for each, the possible isomorphism classes for $F\in (V,f)^\perp_j$.
\begin{figure}[h]
 \small
\qquad $\begin{array}{c|c|c}\hline
 f&F\in ((V,f)^\perp)_j& F \text{ up to isomorphism}\\\hline
 yz(y-z)&\langle  Y^{j},(Y+Z)^{[j]},Z^j\rangle&Y^j, Y^j+aZ^j,Y^j+aZ^j+b(Y+Z)^{[j]}\\
y^2z&\langle  Z^{j},YZ^{j-1},Y^j\rangle&Z^j, YZ^{j-1},Y^j, Y^j+Z^j, Y^j+YZ^{j-1}\\
y^3&\langle  Z^{j},YZ^{j-1},Y^2Z^{j-2}\rangle&Z^j,YZ^{j-1},Y^2Z^{j-2},Y^2Z^{j-2}+Z^j
\\\hline
 \end{array}$\vskip 0.2cm\par
 \qquad\quad 
 \caption{\small Isom. classes of $(f,F)$ for $A \in G_V(T), V=\langle x^2,xy,xz\rangle$.}\label{2aaequivfig}
 (Theorem~\ref{2aAthm}, Equation \eqref{fFeqn}).
 \end{figure}\par
\normalsize
 \begin{theorem}\label{2aAthm}[$I_2$=\#2a] Let $T=(1,3^k,1), k\ge 3$ and assume that $\cha {\sf k}\not=2,3$.  There are the following isomorphism classes and dimensions for algebras $A=R/I\in G_V(T)$ where
 $V=I_2\cong  \langle x^2,xy,xz\rangle$ (\#2a) is fixed  and $k\ge 3$, as in Figures  \ref{2aaequivfig},\ref{2aequivdimfig}, \ref{2aequivfig}.  Each ideal $I$ defining $A\in G_V(T)$ satisfies 
 \begin{align}
 I&=(V,f,\Ann F\cap R_j,\m^{j+1}) \notag\\
 &\text { where $f\in R_3/xR_2\cong R^\prime_3$, and $F\in (V,f)_j^\perp\subset S^\prime_j$},
 \label{fFeqn}
 \end{align}
  where $f\cong  yz(y-z)$ or $y^2z$ or $z^3$.
  Here the fiber $G_V(T)$ over $V$ is irreducible of dimension five: it is fibered  over $V$ by a projective plane $\mathbb P^2$ parametrizing the element $F$, over the 3-space $\mathbb P^3$ parametrizing $f\in R^\prime_3$.
  \par
  There are over $f=yz(y-z)$ a two-parameter family $I(a,b)$ of isomorphism classes (Equations \eqref{Wab2eq} and \eqref{dualgenIabeq}), and a one-parameter family $I(a)$ (Equation \eqref{waeqn}) as well as the specific class $I(0,0)$.  Over each of $y^2z$ and $ z^3$ there are a finite set of isomorphism classes, as listed in Figure \ref{2aequivfig}.
 \end{theorem}

 \begin{proof}  Fix $V=I_2=(x^2,xy,xz)$. 

  For $i\ge 2$ we have $R/(I_2)_i\cong  R^\prime_i$ and $(I_2^\perp)_i=S^\prime_i$. 
Thus, an ideal $I$ defining $A\in G_V(T)$  satisfies
 \begin{equation}
 I=(I_2,f,W,\m^{j+1}),
 \end{equation}
 where we may regard $f\in R_3/xR_2$, parametrized by a projective space $\mathbb P(R^\prime_3)\cong\mathbb P^3$, the two-dimensional $W$  $\Grass(R_j/(I_2,f)_j, 2)$, parametrized by a projective plane $\mathbb P^2$, so a total fiber dimension of five; here also $W=\Ann F\cap R^\prime_j/(V,f)_j$, where $F\in S^\prime_j\cap ((V,f)_j^\perp$.
  The dimension of $G_{\{V\}}(T)=G_{\#2a}(T)$ is thus $7$. \par
  Consider $f\in R_3/xR_2\cong R^\prime_3$: up to isomorphism $f$ is one of $yz(y-z),y^2z$ or $y^3$, respectively, families of dimension $3,2,1$, respectively. For $f=y^2z$ or $y^3$ we will use $F$ to parametrize $I_j=\Ann(F)_j$. In Figure \ref{2aaequivfig} we write the space $\langle (f)^\perp_j\rangle$ and then the isomorphism classes of elements $F$ from that space.  However, for $f=yz(y-z)$ we prefer the direct calculation of isomorphism classes of $W$, just below.\vskip 0.2cm
 \noindent
 {\bf Calculating $W$ when $f=yz(y-z)$.}  Consider the vector subspace $\mathcal W=\langle y^j, yz^{j-1},z^j\rangle \subset R_j$, which induces a basis for $R_j/((V,f)\cap R_j)$.\par
 We may write $W=\langle f_5,f_6\rangle$.
 Then $I=(V,yz(y+z),W,\m^{j+1})$ where $\m=(x,y,z)$ and 
 \begin{align}
 f_5&=a_1y^j+a_2yz^{j-1}+a_3z^j\notag\\
f_6&=b_1y^j+b_2yz^{j-1}+b_3z^j\label{W2eqn}
 \end{align}
 \begin{enumerate}[(i).]
 \item
 If $a_1\not=0$ we may write $f_6=b_2yz^{j-1}+b_3z^j$; if now also $b_2\not=0$, we may, after scaling, write the two parameter family of isomorphism classes $I(a,b)=(xR_1, yz(y-z),W(a,b))$ where
 \begin{equation}
 W(a,b)=\langle f_5,f_6\rangle \text{ where } f_5=y^j-az^j, f_6=yz^{j-1}+bz^j.\label{Wab2eq}
 \end{equation}
 However, we may also apply the involution $ \sigma: y\to z, z\to y$. Then if $a\not=0$ we have
 $\sigma(f_5)=ay^j-z^j$ which we may scale to $y^j-a^{-1}z^j$ and $\sigma(f_6)=zy^{j-1}+by^j=yz^{j-1}+a^{-1}bz^j$, so we have for $a\not=0$
 \begin{equation}
 W(a,b)=W(a^{-1},a^{-1}b).\label{Wabsigmaeq}
 \end{equation}
 Using a dual generator $F(a,b)=aY^j+Z^j-bYZ^{j-1}$, we may write \normalsize
 \begin{equation}
 I(a,b)=( xR_{1},f,\langle\Ann\, F(a,b)\cap R^\prime_j\rangle,\m^{j+1}).\label{dualgenIabeq}
 \end{equation}
 \item If $a_1\not=0$ but $b_2=0$, we may reduce to 
 \begin{equation}
 W(a)=\langle f_5,f_6\rangle, f_5=y^j+ayz^{j-1}, f_6=z^j,\label{waeqn}
 \end{equation} a one-parameter family of isomorphism classes with degree-$j$ dual generator $F=aY^j-YZ^{j-1}$.
 \item If $a_1=b_1=0$, then, since $W$ is two-dimensional, we must have $W=\langle yz^{j-1},z^j\rangle$ isomorphic to $\sigma(W)=\langle y^{j-1}z,y^j\rangle\cong \langle y^j, yz^{j-1}\rangle=W(0,0)$.
 \end{enumerate}
 This completes the calculation of the isomorphism classes for $W$ when $f=yz(y-z)$.
		\vskip 0.2cm
 \par In Figure \ref{2aequivdimfig}, we give the dimension of these isomorphism classes, given $f$ and the class of a dual generator $F\in S_j$ for $I_j$; and in Figure \ref{2aequivfig} we give these isomorphism classes in terms of their defining ideals. For example, the isomorphism class of the pair $(f=y^3, F=Z^j)$ is determined by a choice of $y\in R_1/\langle x\rangle$ then $Z+aY$ (in place of $Z \in \langle Y,Z\rangle$, each a one-dimensional choice. 
The base dimension of the isomorphism class of $y^2z\in R_3/\langle xR_2\rangle\cong \langle y^2,yz,z^2\rangle$ is $2$ (choose $y,z\in R_1/\langle x\rangle$), giving the dimensions $(3,3,3,4,4)$
 for the respective classes of $F$ in row 2 of Figure \ref{2aequivdimfig}. Likewise, the base dimension of $yz(y-z)$ is 3 (parametrized by an element of $R_3/xR_2\cong {\sf k}^4$ up to ${\sf k}^\ast$ action, so by $\mathbb P^3$); there are at most $3(j-2)$, and generically $3(j-2)$ perfect powers\footnote{This multiplicity follows from the theory of Wronskians - see \cite[Definition 2.5 and Lemma 2.6]{IY}: when $K$ is a vector space in $k[Y,Z]$ of dimension 3 and degree j, the Wronskian $W(K)$ is a form of degree $3(j-2)$ in $Y,Z$, whose roots - if distinct -determine the perfect powers in $K$.} in the space $(V,yz(y-z))^\perp_j$, yielding dimension $3$ - the base dimension - when $F=Z^j$, as there are no parameters for the fibre. The generic element of $G_V(T)$ is, evidently the family with two-dimensional fiber $I(a,b)=(V,yz(y-z),W(a,b))$ of Equation \eqref{dualgenIabeq}, 
where  $I(a,b)_j=\langle\Ann\,F(a,b)\cap R_j\rangle, $ with $F(a,b)=aY^j+Z^j-bYZ^{j-1}$. This family has total dimension $7$.  \vskip 0.2cm \label{endproofthm5}
This completes the proof of Theorem 5.\qquad\qquad\qquad $\square$

  \begin{figure}[h]
 \small
 \quad $\begin{array}{c|c|c}\hline
 f&F& \text { dim. of fiber $G_V(T)$ over $V$}\\\hline
 yz\ell&Z^j, Y^j+aZ^j,aY^j+Z^j-bYZ^{j-1}&3,4,5;\\
 y^2z&Z^j, YZ^{j-1},Y^j, Y^j+Z^j, YZ^{j-1}+Y^j&3,3,3,4,4;\\
 y^3&Z^j,YZ^{j-1},Y^2Z^{j-2},Y^2Z^{j-2}+Z^j&2,2,2,3.
 \\\hline
 \end{array}$\vskip 0.2cm\par
\qquad where $T=(1,3^k,1), k\ge 3$ and $ \ell=y-z$.
 \caption{Dimension of isom. classes in $G_V(T),V=\langle x^2,xy,xz\rangle$.}\label{2aequivdimfig}$T=(1,3^k,1), k\ge 3$  (Theorem \ref{2aAthm})
 \end{figure}
 
  \begin{figure}[h]
 \small
 \qquad $\begin{array}{|c||c|c|c|}\hline
 f&F_1&F_2&F_3\\\hline
 yz(y-z)&Z^j &Y^j+aZ^j,&aY^j+Z^j-bYZ^{j-1}\\
 &W(0,0)&W(a)&W(a,b) \text{ of Equation \eqref{Wab2eq}}\\\hline
 \end{array}$\vskip 0.2cm
\qquad$\begin{array}{|c||c|c|c|c|c|}\hline
f&F_1&F_2&F_3&F_4&F_5\\\hline
 y^2z&Z^j& YZ^{j-1}&Y^j& Y^j+Z^j& Y^j+YZ^{j-1}\\
 &(y^j,yz^{j-1})&(y^j,z^j)&(y^j,yz^{j-1})&(y^j-z^j,yz^{j-1})&(y^j-yz^{j-1},z^j)\\\hline
 \end{array}$\vskip 0.2cm
\qquad $\begin{array}{|c||c|c|c|c|}\hline
 f&F_1&F_2&F_3&F_4\\\hline
 y^3&Z^j&YZ^{j-1}&Y^2Z^{j-2}&Y^2Z^{j-2}+Z^j\\\hline
 &(y^2z^{j-2},yz^{j-1})&(y^2z^{j-2},z^j)&(yz^{j-1},z^j)&(y^2z^{j-2}-z^j,yz^{j-1})
 \\\hline
 \end{array}$\vskip 0.2cm\par
\qquad where $T=(1,3^k,1), k\ge 3$ and $ \ell=y+z$.
 \caption{Isom. classes of ideals in $G_V(T), V=\langle x^2,xy,xz\rangle$.}\label{2aequivfig}$T=(1,3^k,1), k\ge 3$ (Theorem \ref{2aAthm}).
 \end{figure}
 \begin{figure}
 \begin{small}
$$
\vbox{\offinterlineskip 
\halign{\strut\hfil# \ \vrule\quad&# \ &# \ &# \ &# \ &# \ &# \
&# \ &# \ &# \ &# \ &# \ &# \ &# \
\cr
total&1&6&8&3\cr
\noalign {\hrule}
0&1 &--&--&--&\cr
1&--&3 &3&1 &\cr
2&-&1 &1&&\cr
3&-&2&3& 1 &\cr
4&-&&1& 1 &\cr
\noalign{\smallskip}
}}
$$
\end{small}
\vskip -0.8cm
\caption{Betti diagram J$_4, V=$ \#2a, various $(f,F)$.}\label{BettiJIfig}
For $f,F)=(y^3,YZ^3)$, also $(f=yz(y-z),$all $F$), and $(f,F)=(y^2z,F\not=Z^4)$. For 
J$_j$ move the last two rows to indices $(j-1,j)$. See Figure \ref{2aequivfig}.
\end{figure}
\begin{figure}
 \begin{small}
$$
\vbox{\offinterlineskip 
\halign{\strut\hfil# \ \vrule\quad&# \ &# \ &# \ &# \ &# \ &# \
&# \ &# \ &# \ &# \ &# \ &# \ &# \
\cr
total&1&7&10&4\cr
\noalign {\hrule}
0&1 &--&--&--&\cr
1&--&3 &3&1 &\cr
2&-&1 &1&&\cr
3&-&2&4& 2 &\cr
4&-&1&2& 1 &\cr
\noalign{\smallskip}
}}
$$
\end{small}
\vskip -0.8cm
\caption{Betti diagram K$_4, V=$ \#2a., for $(f,F)=(y^3,Z^4)$.}\label{BettiKfig}
The lex ideal for $T=(1,3,3,3,1)$) and $(f,F)=(y^2z,Z^4)$.
 For the lex table
K$_j$ of socle degree $j\ge 4$ move the last two rows to indices $(j-1,j)$.
\end{figure}
 \normalsize
 \end{proof}\noindent
 {\it Note}.
  The Betti tables J and K for these classes of Figure \ref{2aequivfig} were calculated using Macaulay 2 for $j=4$. 
  Evidently Betti table $K$ of Figure~\ref{BettiKfig} is the table for the lexicographic ideal (having highest column sums). Note that the lex ideal of socle degree $4$, so $T=(1,3^3,1)$ is that corresponding to $(y^3,Z^4)$ in
  Figure \ref{2aequivdimfig}.
 \newpage \subsection{Questions and Problems}

 \subsubsection{Limits of Betti tables in $\Gor(T)$ and $G_T$.}\label{bettisec}
Let $T$ be an arbitrary Gorenstein sequence - one possible for a graded Artinian Gorenstein algebra. For $\Gor(T)$ in codimension three, S.J. Diesel and D. Buchsbaum-D. Eisenbud determined that, given $T$ the minimal Betti resolutions compatible with $T$ have a poset structure, and satisfy a frontier property: the Zariski closure of a Betti stratum of $\Gor(T)$ is the union of Betti strata (\cite{Di}, see also  (\cite[Theorem 5.25]{IK} for a statement). For our special $T=(1,3^k,1)$ the Gorenstein minimal resolution table is unique if $k>2$ (Betti table C$_j$, Figure \ref{BettiCfig}); there are two Gorenstein minimal resolutions when $k=2$ (Betti table CI - Figure~\ref{BettiCIfig}, and Betti table C).  H. Schenck et al in \cite[p. 65-66]{KKRSSY} show that in an analogous codimension four case, the Betti tables do {\it not} satisfy a frontier property.  I. Peeva has shown that all Betti tables for a given Hilbert function may be obtained by consecutive cancellation from that of a lex-segment ideal (Lemma \ref{Peevalem}). She also showed \cite[Remark after Theorem 1.1]{Pe}:
\begin{lemma}\label{conseclem} Consider a flat family $\mathcal F(t)$ over ${\sf k}[t]$ of graded Artinian quotients of $R$ whose fiber $\mathcal F(0)$ over $0$ is $S/J$ and whose fiber $\mathcal F(t)$ over $t\in {\sf k} \backslash 0$ is $S/J^\prime$. Then the graded Betti numbers $\beta^\prime_{i,j}=\beta_{i,j}(S/J^\prime)$ of the general fiber may be obtained from the graded Betti numbers $\beta_{i,j}(S/J)$ of the special fiber
 by a sequence of consecutive cancellations.
\end{lemma}\par\noindent
An example is the  C,D,E Betti diagrams (Figures \ref{BettiCfig}, \ref{BettiDfig} and \ref{BettiEfig}, respectively) for the three isomorphism classes \#6a.i,,\#6a.ii., and \#6a.iii., respectively where $I_2=(xy,xz,yz)$ (\#6a), for 
$G_T,T=(1,3,3,1)$: each is obtained by consecutive cancellation from the next.\footnote{We thank Alexandra Seceleanu who pointed out an error in a previous writing of the Betti diagram for \#6a.ii.; we have made consequent changes here.}\par

\vskip 0.2cm\par\noindent
{\bf Question}. Lemma \ref{conseclem} assumes a ``jump specialization'' - there is a single isomorphism class among the more general algebras.  Is there a similar statement when we allow the more general family to not be in a single isomorphism class of algebras:  that is, a family $\mathcal F(t)$ whose fiber over $t\in {\sf k}$ is $\mathcal F(t)=S/J^\prime(t)$, but which have a common Betti table $\beta^\prime$ for $t\not=t_0$? Then is $\beta^\prime$ obtained by consecutive cancellation from the Betti table $\beta(S/J(t_0))$?

\begin{question} 
 Do the Betti table strata for $T=(1,3^k,1)$ satisfy the frontier property? If so, is this a more general result 
 for $T$ that are codimension three Gorenstein sequences?
\end{question}
Since we have shown that $\overline{Gor(T)}$ need not be all of $G_T$ for some codimension three
Gorenstein sequences, this question appears to be open for such $T$ in general.

 \subsubsection{Specialization of $G_T$ isomorphism classes.}
Note: As the dimension of the classes $G_V(T)$ of Artinianalgebras corresponding to the nets \#7b,\#7c are seven, we can expect that they each specialize to a dimension six class,\#6c.ii, or \#6c.iii? (latter, when $\cha {\sf k}=3$). Which specializations occur? 

\begin{question} We found in specializing nets of conics, that because each family but \#8b was a complete isomorphism class,  there were no jump specializations, that is if we have specializations $U\to Z$ and the pair $U\to Y, Y\to Z$ then the specialization $U\to Z$ comes from first specializing $U\to Y$, then $Y\to Z$. Now we would need to consider - and perhaps rule out - that the algebra class say in \#8a, might jump, that is, specialize to a class in \#6a, but not specialize there through a class in \#7a. Can we rule such out? 
\end{question}
   \vskip 0.2cm\noindent
   {\bf Goal.} Make tables of specializations of Artinian algebras of Hilbert function $T=(1,3^k,1)$, extending Figure \ref{1fig} for nets of conics. The first table would be for isomorphism classes of  $T=(1,3,3,1)$ algebras, and the second for isomorphism classes of $T=(1,3^k,1), k\ge 3$ algebras.
   
\begin{remark}\label{specializeto6arem} The only direct specializations to \#6a algebras can be from \#7a algebras: here \#7a has two classes, one of dimension 8, and the other of dimension 7. In specializing we must drop dimension of a family by at least one, so we have that \#6a.i. cannot be the specialization of any family in \#7a; \#6a.ii could only be a specialization from \#7a.i of dimension 8; and \#6a.iii might be a specialization of either class in \#7a. 
\end{remark}\vskip 0.2cm\noindent
{\bf Question}.  According to Remark \ref{closurerem} we expect that $\overline{G_{\#8a}(T)}\cap G_{\#7a}(T)=G_{\#7a.ii}(T)$, the dimension seven component. What is the intersection of the closure of a 8-dimensional cell of \#8a with the 8-dimensional cell here?
 \vskip 0.2cm\par\noindent
 \subsubsection{Possible future directions involving Jordan type or applications.}\par
 The Jordan degree type of multiplication $m_\ell$ on a graded Artinian algebra $A$, adds to the information of the lengths of Jordan strings, their initial degrees.
     Note that \cite{AAIY} has a table with the classification of Jordan degree types for graded AG algebras of Hilbert function  $(1,3^k,1)$: these are all symmetric, by a general result of T. Harima and J. Watanabe \cite{HW} (see also \cite[Prop. 2.38]{IMM}).
     But their  limit algebras may have JDT that are not symmetric. The notation $1_{2\uparrow 3}$ means two
      strings of length 1, beginning in degree 2 and 3, respectively.
    \begin{example}[Limit algebra has non-symmetric JDT, using non-generic $\ell$] Take $I=(x^2,xy,xz,y^3,z^3,\m^4)$  of Hilbert function $H(R/I)=(1,3,3,1)$ and $\ell=x$, then $x^2=0$ and the Jordan degree type $S_{x,A}=(2_0,(1_{1\uparrow 2})^2, 1_{2\uparrow 3}$) which is not symmetric.
    \end{example} 
  A.  What are the possible generic JDT for these algebras in $G_T$ that are not Gorenstein?\par
  B. Can we specify how Jordan types correspond to the Betti tables, as in a recent paper of N. Abdallah and H. Schenck \cite{AbSc}?

  \par  C. Some people have applied nets of conics. Can we apply these algebras?
    Mapping germs? \par
    See Appendix B, p. 82 of \cite{AEI} for some applications others have made.
    \par
    D.  The classification of pencils of conics up to isomorphism is well known and classical (see, for example \cite[\S 3, Table 2]{AEI}). Describe the related isomorphism classes in the family $G_T$ of graded Artinian algebras of Hilbert function $T=(1,3,4^k,3,1)$?\par
        \vskip 0.3cm
 
\addcontentsline{toc}{section}{List of Figures}
  \listoffigures
  \vskip 0.5cm
  
    \begin{ack} We appreciate comments of Pedro Macias Marques, Hal Schenck, Alexandra Secealanu, and Maria Evelina Rossi. We appreciate a careful reading and very helpful questions and comments of the referee. The first author would like to thank GS Magnusons Fond for financing research trips to University of C\^{o}te-D'Azur in Nice and to Northeastern University in Boston. 
    \end{ack}
\vskip 0.5cm
  {\bf Declarations:}\par

{\bf Funding}:  The first author's visit to University of C\^{o}te-D'Azur in Nice was funded by GS Magnusons Fond MG2018-0001; her visit to Northeastern University was funded by GS Magnusons Fond MF2022-0011. The authors received no other funding.\vskip 0.3cm
{\bf Author contribution}: All listed authors contributed significantly to the article.\vskip 0.2cm
{\bf Conflicts of Interest/Competing Interests}: The authors have no relevant financial or non-financial interests to disclose.\vskip 0.2cm
{\bf Data Availability}.
All data generated or analysed during this study are included in this published article.
    \vskip 3cm
  \quad {\sf e-mail}:\par
    nancy.abdallah@hb.se\par
    a.iarrobino@northeastern.edu\par
    joachim.yameogo@univ-cotedazur.fr
\end{document}